\documentclass[11pt]{amsart}
\usepackage[utf8]{inputenc}
\usepackage{amsfonts, amssymb, amsmath, amsthm, color, float,enumerate}
\usepackage{url}
\usepackage[unicode,psdextra]{hyperref}
\usepackage{pgfplots,tikz}

\title[Boundedness of multilinear commutators of fractional integral operators]{Some extensions of classes involving pair of weights  related to the boundedness of multilinear commutators associated to generalized fractional integral operators}
\author{}
\date{}

\theoremstyle{plain}
   \newtheorem{teo}{Theorem}
   \newtheorem{coro}[teo]{Corollary}
   \newtheorem{lema}[teo]{Lemma}
   \newtheorem{propo}[teo]{Proposition}
   
\theoremstyle{definition}
   
\theoremstyle{remark}
 \newtheorem{obs}{Remark}

\numberwithin{equation}{section}
\numberwithin{teo}{section}
\allowdisplaybreaks

\definecolor{aquamarine}{rgb}{0.5, 1.0, 0.83}
\definecolor{americanrose}{rgb}{1.0, 0.01, 0.24}
\definecolor{arsenic}{rgb}{0.23, 0.27, 0.29}
\definecolor{blizzardblue}{rgb}{0.67, 0.9, 0.93}
	\definecolor{blush}{rgb}{0.87, 0.36, 0.51}
	\definecolor{celestialblue}{rgb}{0.29, 0.59, 0.82}
	\definecolor{chocolate(web)}{rgb}{0.82, 0.41, 0.12}

\hypersetup{
	colorlinks = true,
	linkcolor = celestialblue,
	anchorcolor = blue,
	citecolor = blush,
	filecolor = blue,
	urlcolor = chocolate(web)
}

\newcounter{BPR}

\begin{document}

	\author[F. Berra]{Fabio Berra}
	\address{Fabio Berra, CONICET and Departamento de Matem\'{a}tica (FIQ-UNL),  Santa Fe, Argentina.}
	\email{fberra@santafe-conicet.gov.ar}
	
	\author[G. Pradolini]{Gladis Pradolini}
	\address{Gladis Pradolini, CONICET and Departamento de Matem\'{a}tica (FIQ-UNL),  Santa Fe, Argentina.}
	\email{gladis.pradolini@gmail.com}
%
	\author[J. Recchi]{Jorgelina Recchi}
	\address{Jorgelina Recchi, Departamento de Matemática, Universidad Nacional del Sur (UNS), Instituto de Matematica (INMABB), Universidad Nacional del Sur-CONICET,  Bahía Blanca, Argentina.}
	\email{jrecchi@gmail.com}
	
	\thanks{The author were supported by CONICET, UNL, ANPCyT and UNS}
	
	

\maketitle

	\begin{abstract}
	We deal with the boundedness properties of higher order
commutators related to some generalizations of the multilinear fractional integral operator of order $m$, $I_\alpha ^m$,
from a product of weighted Lebesgue spaces into adequate weighted Lipschitz spaces, extending some previous estimates for the linear case. Our study includes two different types of commutators and sufficient conditions on the weights in order to guarantee the continuity properties described above. We also exhibit the optimal range of the parameters involved. The optimality
is understood in the sense that the parameters defining the corresponding spaces belong to a
certain region, being the weights trivial outside of it. We further show examples of weights for the class which cover the mentioned
area.
	\end{abstract}

\section{Introduction}\label{seccion: introduccion}

Many classical operators in Harmonic Analysis whose continuity properties were extensively studied, have shown to behave in a suitable way when their multilinear versions act in the corresponding multilinear spaces. For example, in \cite{L-O-P-T-T} the authors proved that both, the multilinear Calderón-Zygmund operators and their commutators with \rm{BMO} symbols, are bounded from a product of weighted Lebesgue spaces to an associated weighted Lebesgue space, with weights belonging to the $A_{\vec{p}}$ multilinear class. This article established the starting point of the weighted theory in this general context. Motivated by the result in \cite{L-O-P-T-T}, similar estimates for multilinear fractional integral operators were obtained in \cite{Moen09}, and in \cite{CX10} for their commutators with \rm{BMO} symbols. The weights involved in both articles belong to a class that generalizes those given in \cite{L-O-P-T-T} and in \cite{Muckenhoupt-Wheeden74} for the linear case.

Regarding the two-weighted theory, also in \cite{Moen09} K. Moen gave a complete discussion showing sufficient bump conditions in order to guarantee the continuity properties of the multilinear fractional integral operator acting between a wider class of Lebesgue spaces than those obtained in the one-weight theory. The obtained results generalize the linear version given, for example, in \cite{SW92} and \cite{P94pot}.

In \cite{BPR22} the authors studied the two-weighted boundedness properties of the multilinear fractional integral operator between a product of weighted Lebesgue spaces into appropiate weighted Lipschitz spaces associated to a parameter $\delta$. They characterized the classes of
weights related to this problem, showing also the optimal range of the
parameters involved. The optimality is understood in the sense that the parameters defining
the corresponding spaces belong to a certain region, being the weights trivial outside of it. These results extend the corresponding proved in \cite{Pradolini01} for the linear case.
 
Our main interest in this article is the study of the boundedness properties for higher order commutators related to some generalizations of the multilinear fractional integral operator of order $m$, $I_\alpha^m$, defined by
	\[I_\alpha^m \vec{f}(x)=\int_{(\mathbb{R}^n)^m} \frac{\prod_{i=1}^m f_i(y_i)}{(\sum_{i=1}^m|x-y_i|)^{mn-\alpha}}\,d\vec{y},\]
	where $0<\alpha<mn$, $\vec{f}=(f_1,f_2,\dots, f_m)$ and $\vec{y}=(y_1,y_2,\dots, y_m)$. These operators are given by
	\[T_\alpha^m\vec{f}(x)=\int_{(\mathbb{R}^n)^m} K_\alpha(x,\vec{y})\prod_{i=1}^m f_i(y_i)\,d\vec{y},\]
	where $K_\alpha$ satisfies certain size  and regularity conditions (see Section~\ref{seccion: preliminares}). Particularly, the size condition allows us to conclude that $|T_\alpha^m\vec{f}|\leq CI_\alpha^m \vec{g}$, where $\vec{g}=(|f_1|,\dots,|f_m|)$. This estimation guarantees the  boundedness  of $T_\alpha^m$ with different weights, between a product of  Lebesgue spaces into a related Lebesgue space when the class of multilinear weights involved is an extension of that given in \cite{Moen09} for the continuity properties of $I_\alpha^m$.
	Nevertheless, this argument cannot be applied for other spaces such that $\mathrm{BMO}$, Lipschitz or  Morrey because of the fact that they have not a growth property such as Lebesgue spaces. So, we shall focus our attention in  studying the boundedness properties of the commutators of $T_{\alpha}^m$ acting between a product of weighted Lebesgue spaces into certain weighted versions of the aforementioned spaces. In the linear case this problem was studied in \cite{PRR21} for higher order commutators including the case of the operator $I_\alpha$, which had been previously given in  \cite{Pradolini01} for the two-weighted setting in the linear context. For similar problems involving other weighted type of Lipschitz spaces see \cite{HSV}, \cite{Prado01cal} and \cite{PR}.  On the other hand, in the multilinear case and for $T_\alpha^m=I_{\alpha}^m$ (that is $K_\alpha(x,\vec{y})=(\sum_{i=1}^m|x-y_i|)^{\alpha-mn}$), generalizations of two-weighted problems can be found in \cite{BPR22} and \cite{BPR22(2)} (see also \cite{AHIV} for the unweighted problem).
	
	 In this paper we study the boundedness of commutators of fractional operators, including  $I_{\alpha}^m$, between a product of weighted Lebesgue spaces and weighted generalizations of those introduced in \cite{Pee69}. For a weight $w$, the latter are denoted by  $\mathbb{L}_w(\delta)$ and defined as the collection of locally integrable functions $f$ for which the inequality 
	\[\frac{\|w\mathcal{X}_B\|_\infty}{|B|^{1+\delta/n}}\int_B|f(x)-f_B|\,dx\leq C\]
	holds for every ball $B\subset \mathbb{R}^n$, where $f_B=|B|^{-1}\int_B f$.

	We shall consider two different types of higher order commutators of $T_\alpha^m$.
	Given an $m$-tuple of functions $\mathbf{b}=(b_1,\dots,b_m)$, where each component belongs to $L^1_{\rm{loc}}$, the first multilinear commutator we shall be dealing with,  $T_{\alpha,\mathbf{b}}^m$, is given by the expression
\[T_{\alpha,\mathbf{b}}^m\vec{f}(x)=\sum_{j=1}^m T_{\alpha,b_j}^m\vec{f}(x),\]
where $T_{\alpha,b_j}^m$ is formally defined by
\[T_{\alpha,b_j}^m\vec{f}(x)=\int_{(\mathbb{R}^n)^m}(b_j(x)-b_j(y_j))K_\alpha(x,\vec{y})\prod_{i=1}^m f_i(y_i)\,d\vec{y}.\]
On the other hand, the second type of commutator that we shall consider can be expressed as
\[\mathcal{T}_{\alpha,\mathbf{b}}^m\vec{f}(x)=\int_{(\mathbb{R}^n)^m}K_\alpha(x,\vec{y})\prod_{i=1}^m(b_i(x)-b_i(y_i))f_i(y_i)\,d\vec{y}.\]
The last two equalities above are a consequence of their definitions given in Section~\ref{seccion: preliminares}, and both commutators were introduced in \cite{PT03} and \cite{PPTT}, respectively.

We shall also be dealing with multilinear symbols with components belonging  to the classical Lipschitz spaces $\Lambda(\delta)$ (for more information see Section~\ref{seccion: preliminares}).

Let $\vec{p}=(p_1,p_2,\dots,p_m)$ be a vector of exponents such that  $1\leq p_i\leq \infty$ for every $i$. Let $\beta$, $\delta$ and $\tilde\delta$ be real constants. Given $w$, $\vec{v}=(v_1,v_2,\dots,v_m)$ and $\vec{p}$, we say that $(w,\vec{v})\in \mathbb{H}_m(\vec{p}, \beta,\tilde\delta)$ if there exists a positive constant $C$ such that the inequality
	\begin{equation*}
	    \frac{\|w\mathcal{X}_B\|_\infty}{|B|^{(\tilde \delta-\delta)/n}}\prod_{i=1}^m\left(\int_{\mathbb{R}^n}\frac{v_i^{-p_i'}}{(|B|^{1/n}+|x_B-y|)^{(n-\beta_i+\delta/m)p_i'}}\,dy\right)^{1/p_i'}\leq C
	\end{equation*}
	holds for every ball $B=B(x_B,R)$, with the obvious changes when $p_i=1$. The numbers $\beta_i$ satisfy $\sum_{i=1}^m \beta_i=\beta$ and also $0<\beta_i<n$, for every $i$ (see Section~\ref{seccion: preliminares} for further details related to these classes of weights).  
	
We shall now state our main results. From now on, $1/p=\sum_{i=1}^m1/p_i$. See Section~\ref{seccion: preliminares} for details.
	
\begin{teo}\label{teo: acotacion Lp Lipschitz para T_alpha,b (suma)}
Let $0<\alpha<mn$ and $T_\alpha^m$ be a multilinear fractional operator with kernel $K_\alpha$ satisfying \eqref{eq: condicion de tamaño} and \eqref{eq: condicion de suavidad}. Let $0<\delta<\min\{\gamma,mn-\alpha\}$, $\tilde\alpha=\alpha+\delta$ and $\vec{p}$ a vector of exponents that satisfies $p>n/\tilde\alpha$. Let  $\mathbf{b}=(b_1,\dots,b_m)$ be a vector of symbols such that $b_i\in\Lambda(\delta)$, for $1\leq i \leq m$. Let $\tilde\delta\leq \delta$ and  $(w,\vec{v})$ be a pair of weights belonging  to the class $\mathbb{H}_m(\vec{p},\tilde\alpha,\tilde\delta)$ such that $v_i^{-p_i'}\in\mathrm{RH}_m$, for every $i$ such that $1<p_i\leq \infty$. Then the  multilinear commutator $T_{\alpha,\mathbf{b}}^m$ is bounded from $\prod_{i=1}^mL^{p_i}(v_i^{p_i})$ to $\mathbb{L}_w(\tilde\delta)$, that is, there exists a positive constant $C$ such that the inequality
\[\frac{\|w\mathcal{X}_B\|_\infty}{|B|^{1+\tilde\delta/n}}\int_B |T_{\alpha,\mathbf{b}}^m\vec{f}(x)-(T_{\alpha,\mathbf{b}}^m\vec{f})_B|\,dx\leq C\prod_{i=1}^m\|f_iv_i\|_{p_i}\]
holds for every ball $B$ and every $\vec{f}$ such that $f_iv_i\in L^{p_i}$, for $1\leq i\leq m$.
\end{teo}

Concerning the commutator $\mathcal{T}_{\alpha,\mathbf{b}}^m$ we have the following result.

\begin{teo}\label{teo: acotacion Lp Lipschitz para T_alpha,b (producto)}
Let $0<\alpha<mn$ and $T_\alpha^m$ be a multilinear fractional operator with kernel $K_\alpha$ satisfying \eqref{eq: condicion de tamaño} and \eqref{eq: condicion de suavidad}. Let  $0<\delta<\min\{\gamma,(mn-\alpha)/m\}$, $\tilde\alpha=\alpha+m\delta$ and $\vec{p}$ a vector of exponents that satisfies $p>n/\tilde\alpha$. Let  $\mathbf{b}=(b_1,\dots,b_m)$ be a vector of symbols such that $b_i\in\Lambda(\delta)$, for $1\leq i \leq m$. Let $\tilde\delta\leq \delta$ and  $(w,\vec{v})$ be a pair of weights belonging  to the class $\mathbb{H}_m(\vec{p},\tilde\alpha,\tilde\delta)$ such that $v_i^{-p_i'}\in\mathrm{RH}_m$, for every $1<p_i\leq \infty$. Then the  multilinear commutator $\mathcal{T}_{\alpha,\mathbf{b}}^m$ is bounded from $\prod_{i=1}^mL^{p_i}(v_i^{p_i})$ to $\mathbb{L}_w(\tilde\delta)$, that is, there exists a positive constant $C$ such that the inequality
\[\frac{\|w\mathcal{X}_B\|_\infty}{|B|^{1+\tilde\delta/n}}\int_B |\mathcal{T}_{\alpha,\mathbf{b}}^m\vec{f}(x)-(\mathcal{T}_{\alpha,\mathbf{b}}^m\vec{f})_B|\,dx\leq C\prod_{i=1}^m\|f_iv_i\|_{p_i}\]
holds for every ball $B$ and every $\vec{f}$ such that $f_iv_i\in L^{p_i}$, for $1\leq i\leq m$.
\end{teo}

Note that the restriction on the parameter $\delta$ implicated is different in each theorem due to the nature of the considered commutators. 

When we deal with $w=\prod_{i=1}^m v_i$, which is the natural substitute of the one weight theory in the linear case, we shall say that $\vec{v}\in \mathbb{H}_m({\vec{p},\beta,\tilde\delta})$. Then we have the following lemma.

\begin{lema}\label{lema: consecuencia de pesos iguales}
     Let $0<\beta<mn$, $\tilde\delta\in\mathbb{R}$,  and $\vec{p}$ a vector of exponents. If $\vec{v}\in \mathbb{H}_m({\vec{p},\beta,\tilde\delta})$, then $\tilde\delta=\beta-n/p$. 
\end{lema}

The proof of this result follows similar arguments as those of Theorem 1.3 in \cite{BPR22} and we shall omit it. As a consequence of this lemma we can prove that if $\tilde\delta<\tau=(\beta-mn)(1-1/m)+\delta/m$, then condition \eqref{eq: condicion H_m(p,alfa,delta)} can be reduced to the class $A_{\vec{p},\infty}$, defined as the collection of multilinear weights $\vec{v}=(v_1,\dots,v_m)$ for which the inequality  
\begin{equation}\label{eq: clase Ap,infinito}
[\vec{v}]_{A_{\vec{p},\infty}}=\sup_{B\subset \mathbb{R}^n} \left\|\mathcal{X}_B\prod_{i=1}^m v_i\right\|_\infty\prod_{i=1}^m\left(\frac{1}{|B|}\int_B v_i^{-p_i'}\right)^{1/p_i'}<\infty
\end{equation}
holds (see Corollary~\ref{coro: suficiencia clase Ap,infinito}).

The paper is organized as follows. In Section~\ref{seccion: preliminares} we give the main definitions required in the sequel. In Section~\ref{seccion: auxiliares} we state and prove some auxiliary results that will be useful for the proof of the main theorems, which are contained in Section~\ref{seccion: conmutador suma} and~\ref{seccion: conmutador producto}. Finally in Section~\ref{seccion: clase de pesos} we prove some properties of the class $\mathbb{H}_m(\vec{p},\beta,\tilde\delta)$ and show the optimality of the associated parameters.

\section{Preliminaries and definitions}\label{seccion: preliminares}

Throughout the paper $C$ will denote an absolute constant that may change in every occurrence. By $A\lesssim B$ we mean that there exists a positive constant $c$ such that $A\leq c B$.  We say that $A\approx B$ when $A\lesssim B$ and $B\lesssim A$. 

By $m\in \mathbb{N}$ we denote the multilinear parameter involved in our estimates. Given a set $E$, with $E^m$ we shall denote the cartesian product of $E$ $m$ times. We shall be dealing with operators given by the expression
\begin{equation}\label{eq: operador T_alpha^m}
    T_{\alpha}^m\vec{f}(x)=\int_{(\mathbb{R}^n)^m}K_\alpha(x,\vec{y})\prod_{i=1}^mf_i(y_i)\,d\vec{y},
\end{equation}
for $0<\alpha<mn$, 
where $\vec{f}=(f_1,\dots,f_m)$, $\vec{y}=(y_1,\dots,y_m)$ and $K_\alpha$ is a kernel satisfying the size condition 
\begin{equation}\label{eq: condicion de tamaño}
    |K_\alpha(x,\vec{y})|\lesssim \frac{1}{(\sum_{i=1}^m|x-y_i|)^{mn-\alpha}}
\end{equation}
and an additional smoothness condition given by
\begin{equation}\label{eq: condicion de suavidad}
    |K_\alpha(x,\vec{y})-K_\alpha(x',\vec{y})|\lesssim \frac{|x-x'|^\gamma}{(\sum_{i=1}^m|x-y_i|)^{mn-\alpha+\gamma}},
\end{equation}
for some $0<\gamma\leq 1$, whenever $\sum_{i=1}^m|x-y_i|>2|x-x'|$. It is easy to check that $T_{\alpha}^m=I_\alpha^m$ defined above, when we consider  $K_\alpha(x,\vec{y})=(\sum_{i=1}^m |x-y_i|)^{\alpha-mn}$. 

We shall introduce two versions of commutators of the operators above. For a specified linear operator $T$ and a function $b\in L^{1}_{\mathrm{loc}}$ we recall that the classical commutator $T_b$ or $[b,T]$ is given by the expression
\[[b,T]f=bTf-T(bf).\]

When we deal with multilinear functions and symbols, it will be necessary to emphasize how we proceed to perform the commutation. If $b\in L^{1}_{\mathrm{loc}}$, $T$ is a multilinear operator and $\vec{f}=(f_1,f_2\dots,f_m)$ we write 
\[[b,T]_j(\vec{f})=bT(\vec{f})-T((f_1,\dots,bf_j,\dots,f_m)),\]
that is, $[b,T]_j$ is obtained by commuting $b$ with the $j$-th entry of $\vec{f}$. 

The first version of the commutator is defined as follows. Given an $m$-tuple $\mathbf{b}=(b_1,\dots,b_m)$, with $b_i\in  L^{1}_{\mathrm{loc}}$ for every $i$, we define the multilinear commutator of $T_{\alpha}^m$ by the expression
\[T_{\alpha,\mathbf{b}}^m\vec{f}(x)=\sum_{j=1}^m T_{\alpha,b_j}^m\vec{f}(x),\]
where
\[T_{\alpha,b_j}^m\vec{f}(x)=[b_j, T_{\alpha}^m]_{j}\vec f(x).\]
 
As a consequence of \eqref{eq: operador T_alpha^m} it is not difficult to see that
\[T_{\alpha,b_j}^m\vec{f}(x)=\int_{(\mathbb{R}^n)^m}(b_j(x)-b_j(y_j))K_\alpha(x,\vec{y})\prod_{i=1}^m f_i(y_i)\,d\vec{y}.\]

We now introduce the second type of commutator of $T_{\alpha}^m$. The
 multilinear product commutator  $\mathcal{T}_{\alpha,\mathbf{b}}^m$ is defined iteratively as follows
 \[\mathcal{T}_{\alpha,\mathbf{b}}^m\vec{f} = [b_m,\dots [b_{2},[b_1, T_\alpha^m]_1]_{2}\dots ]_m\vec{f}.\]
 The expression above does not involve a simple notation, so we shall provide an alternative way to denote this commutator (see Section~\ref{seccion: auxiliares}). 
 
 By means of \eqref{eq: operador T_alpha^m} we can also obtain an integral representation for this operator (see Proposition~\ref{propo: representacion de conmutador producto} below), namely
 \[\mathcal{T}_{\alpha,\mathbf{b}}^m\vec{f}(x)=\int_{(\mathbb{R}^n)^m}K_\alpha(x,\vec{y})\prod_{i=1}^m(b_i(x)-b_i(y_i))f_i(y_i)\,d\vec{y}.\]
 
 By a weight we understand any positive and locally integrable function. 
 
 Given $\delta \in \mathbb{R}$ and a weight $w$ we say that a locally integrable function $f\in  \mathbb{L}_{w}(\delta)$ if there exists a positive constant $C$ such that the inequality 
\begin{equation}\label{eq: definicion clase Lipschitz norma inf}
\frac{\|w\mathcal{X}_B\|_\infty}{|B|^{1+\delta/n}}\int_B|f(x)-f_B|\,dx\leq C
\end{equation} 
holds for every ball $B$, where $f_B=|B|^{-1}\int_B f$. The smallest constant $C$ in \eqref{eq: definicion clase Lipschitz norma inf} will be denoted by $\|f\|_{\mathbb{L}_w(\delta)}$.
If $\delta=0$ the space $\mathbb{L}_{w}(\delta)$ coincides with a weighted version of the BMO spaces introduced in \cite{Muckenhoupt-Wheeden74}, the classical Lipschitz functions  when $0<\delta<1$ and the Morrey spaces when $-n<\delta<0$. These classes of functions were also studied in \cite{Pradolini01}. 
 
Regarding the symbols, we shall be dealing with the $\Lambda(\delta)$ Lipschitz spaces given, for $0<\delta<1$, by the collection of functions $b$ verifying 
\[|b(x)-b(y)|\le C|x-y|^{\delta}.\]
The smallest constant for this inequality to hold will be denoted by $\|b\|_{\Lambda(\delta)}$. For a given $\mathbf{b}=(b_1,\dots,b_m)$, with $b_i\in \Lambda(\delta)$ for every $1\leq i\leq m$, we define
	\[\|\mathbf{b}\|_{(\Lambda(\delta))^m}=\max_{1\leq i\leq m}\|b_i\|_{\Lambda(\delta)}.\]

 Let $S_m=\{0,1\}^m$. Given a set $B$ and $\sigma=(\sigma_1,\sigma_2,\dots,\sigma_m)\in S_m$, we define 
	\[B^{\sigma_i}=\left\{
	\begin{array}{ccl}
	B,&\textrm{ if }&\sigma_i=1\\
	\mathbb{R}^n\backslash B,&\textrm{ if }&\sigma_i=0.
	\end{array}
	\right.\]
	With the notation $\mathbf{B}^\sigma$ we will understand the cartesian product $B^{\sigma_1}\times B^{\sigma_2}\times\dots\times B^{\sigma_m}$. 
	
	For $\sigma=(\sigma_1,\sigma_2,\dots,\sigma_m)$ we also define
	\[\bar{\sigma}_i=\left\{\begin{array}{ccl}
	   1  & \mathrm{ if } & \sigma_i=0;  \\
	   0  & \mathrm{ if } & \sigma_i=1,
	\end{array}
	\right.
	\]
	for every $1\leq i\leq m$ and $|\sigma|=\sum_{i=1}^m\sigma_i$.

	We now describe the classes of weights involved in our estimates. Let $\delta$ be a fixed real constant. If $1\leq p_i\leq \infty$ for every $i$, the $m$-tuple $\vec{p}=(p_1,p_2,\dots,p_m)$   will be called a vector of exponents.  We shall also denote $1/p=\sum_{i=1}^m 1/p_i$. Let $\beta$ and $\tilde\delta$ be real constants. Given $w$, $\vec{v}=(v_1,v_2,\dots,v_m)$ and a vector of exponents $\vec{p}$, we say that $(w,\vec{v})\in \mathbb{H}_m(\vec{p}, \beta,\tilde\delta)$ if there exists a positive constant $C$ such that the inequality
	\begin{equation}\label{eq: condicion H_m(p,alfa,delta)}
	    \frac{\|w\mathcal{X}_B\|_\infty}{|B|^{(\tilde \delta-\delta)/n}}\prod_{i=1}^m\left(\int_{\mathbb{R}^n}\frac{v_i^{-p_i'}}{(|B|^{1/n}+|x_B-y|)^{(n-\beta_i+\delta/m)p_i'}}\,dy\right)^{1/p_i'}\leq C
	\end{equation}
	holds for every ball $B=B(x_B,R)$. The numbers $\beta_i$ satisfy $\sum_{i=1}^m \beta_i=\beta$ and also $0<\beta_i<n$, for every $i$, which leads to $0<\beta<mn$. We shall also see that the parameters $\tilde\delta$ and $\beta$ are related to $\delta$. When $p_i=1$ for some $i$ the integral above is understood as
	\[\left\|\frac{v_i^{-1}}{(|B|^{1/n}+|x_B-\cdot|)^{n-\beta_i+\delta/m}}\right\|_\infty.\]
	Let $\mathcal{I}_1=\{1\leq i\leq m: p_i=1\}$ and $\mathcal{I}_2=\{1\leq i\leq m: 1<p_i\leq \infty\}$. Condition~\eqref{eq: condicion H_m(p,alfa,delta)} implies that
	\begin{equation}\label{eq: condicion local}
	     \frac{\|w\mathcal{X}_B\|_\infty}{|B|^{\tilde\delta/n+1/p-\beta/n}}\prod_{i\in\mathcal{I}_1}\|v_i^{-1}\mathcal{X}_B\|_\infty\prod_{i\in\mathcal{I}_2}\left(\frac{1}{|B|}\int_B v_i^{-p_i'}\right)^{1/p_i'}\leq C
	\end{equation}
	and
	\begin{equation}\label{eq: condicion global}
	    \frac{\|w\mathcal{X}_B\|_\infty}{|B|^{(\tilde \delta-\delta)/n}}\prod_{i\in\mathcal{I}_1}\left\|\frac{v_i^{-1}\mathcal{X}_{\mathbb{R}^n\backslash B}}{|x_B-\cdot|^{n-\beta_i+\frac{\delta}{m}}}\right\|_\infty \,\prod_{i\in\mathcal{I}_2}\left(\int_{\mathbb{R}^n\backslash B}\frac{v_i^{-p_i'}}{|x_B-\cdot|^{(n-\beta_i+\frac{\delta}{m})p_i'}}\right)^{1/p_i'}\leq C.
	\end{equation}
	We shall refer to the inequalities above as the local and global conditions, respectively.

	Furthermore, given $\sigma\in S_m$ we can estimate the $i$-th factor in \eqref{eq: condicion H_m(p,alfa,delta)} depending whether $\sigma_i=0$ or $\sigma_i=1$. Thus we have that condition \eqref{eq: condicion H_m(p,alfa,delta)} implies 
	\begin{equation}\label{eq: condicion mezclada para sigma}
	\frac{\|w\mathcal{X}_B\|_\infty}{|B|^{(\tilde\delta-\delta)/n+\theta(\sigma)}}\prod_{i: \sigma_i=1}\|v_i^{-1}\mathcal{X}_B\|_{p_i'}\prod_{i:\sigma_i=0}\left\|\frac{v_i^{-p_i'}\mathcal{X}_{\mathbb{R}^n\backslash B}}{|x_B-\cdot|^{n-\beta_i+\delta/m}}\right\|_{p_i'}\leq C,
	\end{equation}
	where $\theta(\sigma)=\sum_{i:\sigma_i=1} 1-\beta_i/n+\delta/(mn)$.

We recall that a weight $w$ belongs to the \textit{reverse H\"{o}lder} class $\mathrm{RH}_s$, $1<s<\infty$, if there exists a positive constant $C$ such that the inequality
\[\left(\frac{1}{|B|}\int_B w^s\right)^{1/s}\leq \frac{C}{|B|}\int_B w\]
holds for every ball $B$ in $\mathbb{R}^n$. The smallest constant for which the inequality above holds is denoted by $[w]_{\mathrm{RH}_s}$. It is not difficult to see that $\mathrm{RH}_t\subset \mathrm{RH}_s$ whenever $1<s<t$. We say that $w\in \mathrm{RH}_{\infty}$ if 
\[\sup_B w\le  \frac{C}{|B|}\int_B w,\]
for some positive constant $C$.	

\section{Auxiliary results}\label{seccion: auxiliares}

We devote this section to state and prove some facts that will be useful in our estimates. 

The following proposition establishes an alternative way to denote the product commutator $\mathcal{T}_{\alpha,\mathbf{b}}^m$.

\begin{propo}\label{propo: representacion de conmutador producto}
Let $T_\alpha^m$ be a multilinear operator as in \eqref{eq: operador T_alpha^m} and $\mathbf{b}=(b_1,\dots,b_m)$ where $b_i\in L^1_{\mathrm{loc}}$ for $1\leq i\leq m$. Then we have that
\[\mathcal{T}_{\alpha,\mathbf{b}}^m\vec{f}(x)=\sum_{\sigma\in S_m} (-1)^{m-|\sigma|}\left(\prod_{i=1}^m b_i^{\sigma_i}(x)\right)T_{\alpha}^m(f_1b_1^{\bar\sigma_1},\dots,f_m b_m^{\bar\sigma_m})(x).\]
Furthermore, we have the integral representation
\[\mathcal{T}_{\alpha,\mathbf{b}}^m\vec{f}(x)=\int_{(\mathbb{R}^n)^m}K_\alpha(x,\vec{y})\prod_{i=1}^m(b_i(x)-b_i(y_i))f_i(y_i)\,d\vec{y}.\]
\end{propo}

\begin{proof}
Let us introduce some notation in order to make our calculations simpler. Fix a symbol $\mathbf{b}=(b_1,\dots,b_m)$ and let $F_k$ be the operator resulting after perform $k$ iterative commutings, $1\leq k\leq m$, that is
\[F_k \vec{f} =[b_k,\dots [b_{2},[b_1, T_\alpha^m]_1]_{2}\dots ]_k \vec{f}.\]
Given $\sigma\in S_m$, let us also denote with
\[\vec{g}_{\mathbf{b},\sigma}^{\,k}=(g_1,g_2,\dots,g_m)\]
where
\[g_i=\left\{\begin{array}{ccl}
     f_ib_i^{\bar{\sigma}_i} & \mathrm{ if } & 1\leq i\leq k, \\
     f_i & \mathrm{ if } & k<i\leq m.
\end{array}
\right.
\]
We shall proceed by induction in order to show that
\begin{equation}\label{eq: propo: representacion de conmutador producto - eq1}
F_k \vec{f} = \sum_{\sigma\in S_k} (-1)^{k-|\sigma|}\left(\prod_{i=1}^k b_i^{\sigma_i}\right)T_{\alpha}^m \left(\vec{g}_{\mathbf{b},\sigma}^{\,k}\right).
\end{equation}

Notice that the case $k=1$ is immediate after performing one commutation only. Let us assume that the expression above holds for $k$ and we shall prove it for $k+1$. By the definition and the inductive hypothesis we have
\begin{align*}
    F_{k+1}\vec{f}&=[b_{k+1}, F_k]_{k+1}\vec{f}\\
    &=b_{k+1}\sum_{\sigma\in S_k} (-1)^{k-|\sigma|}\left(\prod_{i=1}^k b_i^{\sigma_i}\right)T_{\alpha}^m \left(\vec{g}_{\mathbf{b},\sigma}^{\,k}\right)\\
    &\qquad-\sum_{\sigma\in S_k} (-1)^{k-|\sigma|}\left(\prod_{i=1}^k b_i^{\sigma_i}\right)T_{\alpha}^m (f_1b_1^{\bar{\sigma}_1},\dots,f_kb_k^{\bar{\sigma}_k},f_{k+1}b_{k+1},f_{k+2},\dots,f_m).
\end{align*}
Observe that $\{\theta\in S_{k+1}\}=\{\theta\in S_{k+1}: \theta_{k+1}=1\}\cup \{\theta\in S_{k+1}: \theta_{k+1}=0\}$. By rewriting the sums above we get
\begin{align*}
    F_{k+1}\vec{f}&=\sum_{\theta\in S_{k+1}, \theta_{k+1}=1}(-1)^{k-(|\theta|-1)}\left(\prod_{i=1}^{k+1}b_i^{\theta_i}\right)T_{\alpha}^m \left(\vec{g}_{\mathbf{b},\theta}^{\,k+1}\right)\\
    &\qquad +\sum_{\theta\in S_{k+1},\theta_{k+1}=0} (-1)^{k+1-|\theta|}\left(\prod_{i=1}^{k+1} b_i^{\theta_i}\right)T_{\alpha}^m\left(\vec{g}_{\mathbf{b},\theta}^{\,k+1}\right)\\
    &= \sum_{\theta\in S_{k+1}}(-1)^{k+1-|\theta|}\left(\prod_{i=1}^{k+1}b_i^{\theta_i}\right)T_{\alpha}^m \left(\vec{g}_{\mathbf{b},\theta}^{\,k+1}\right).
\end{align*}
Therefore \eqref{eq: propo: representacion de conmutador producto - eq1} holds for every $1\leq k\leq m$, and the case $k=m$ allows us to conclude the desired estimate.

In order to show the integral representation, we shall prove that
\begin{equation}\label{eq: propo: representacion de conmutador producto - eq2}
F_k \vec{f} = \int_{(\mathbb{R}^n)^m}K_\alpha(x,\vec{y})\prod_{i=1}^k (b_i(x)-b_i(y_i))\prod_{i=1}^m f_i(y_i)\,d\vec{y}.
\end{equation}
We proceed again by induction on $k$. If $k=1$ we have that
\begin{align*}
F_1\vec{f}(x)&=b_1(x)T_\alpha^m\vec{f}(x)-T_\alpha^m((b_1f_1,f_2,\dots,f_m))(x)\\
&=\int_{(\mathbb{R}^n)^m}K_\alpha(x,\vec{y})(b_1(x)-b_1(y_1))\prod_{i=1}^mf_i(y_i)\,d\vec{y}.
\end{align*}
Now assume the representation holds for $k$ and let us prove it for $k+1$. Indeed, by using the definition of $F_k$ and the inductive hypothesis we get
\begin{align*}
    F_{k+1}\vec{f}(x)&=[b_{k+1}, F_k]_{k+1}\vec{f}(x)\\
    &=b_{k+1}(x)F_k\vec{f}(x)-F_k(f_1,\dots,f_k,f_{k+1}b_{k+1},f_{k+2},\dots,f_m)(x)\\
    &=b_{k+1}(x)\int_{(\mathbb{R}^n)^m}K_{\alpha}(x,\vec{y})\prod_{i=1}^k(b_i(x)-b_i(y_i))\prod_{i=1}^mf_i(y_i)\,d\vec{y}\\
    &\qquad - \int_{(\mathbb{R}^n)^m}K_\alpha(x,\vec{y})b_{k+1}(y_{k+1})\prod_{i=1}^k(b_i(x)-b_i(y_i))\prod_{i=1}^mf_i(y_i)\,d\vec{y}\\
    &=\int_{(\mathbb{R}^n)^m}K_\alpha(x,\vec{y})\prod_{i=1}^{k+1}(b_i(x)-b_i(y_i))\prod_{i=1}^mf_i(y_i)\,d\vec{y}.
\end{align*}
This shows that the representation holds for $k+1$ also. By putting $k=m$ we get the integral representation for $\mathcal{T}_{\alpha,\mathbf{b}}^m$.
\end{proof}

\begin{lema}\label{lema: estimacion diferencia de productos}
Let $m\in\mathbb{N}$ and $a_i,b_i$ and $c_i$ be real numbers for $1\leq i\leq m$. Then
\[\prod_{i=1}^m (a_i-b_i)-\prod_{i=1}^m (c_i-b_i)=\sum_{j=1}^m (a_j-c_j)\prod_{i<j}(a_i-b_i)\prod_{i>j}(c_i-b_i).\]
\end{lema}

\begin{proof}
We proceed by induction on $m$. If $m=1$ it is immediate, both sides are equal to $a_1-c_1$ since the products on the right-hand side are equal to 1.

Assume that the equality holds for $m=k$ and let us prove it for $m=k+1$. We have that
\begin{align*}
\prod_{i=1}^{k+1} (a_i-b_i)-\prod_{i=1}^{k+1} (c_i-b_i)&=(a_{k+1}-b_{k+1})\prod_{i=1}^{k} (a_i-b_i)-\prod_{i=1}^{k+1} (c_i-b_i)\\
&=(a_{k+1}-c_{k+1})\prod_{i=1}^{k} (a_i-b_i)\\
&\qquad +(c_{k+1}-b_{k+1})\left(\prod_{i=1}^{k} (a_i-b_i)-\prod_{i=1}^{k} (c_i-b_i)\right).
\end{align*}
By using the inductive hypothesis we arrive to
\begin{align*}
  \prod_{i=1}^{k+1} (a_i-b_i)-\prod_{i=1}^{k+1} (c_i-b_i)&=(a_{k+1}-c_{k+1})\prod_{i=1}^{k} (a_i-b_i)\\
  &\qquad+(c_{k+1}-b_{k+1})\sum_{j=1}^k (a_j-c_j)\prod_{i<j}(a_i-b_i)\prod_{i>j}(c_i-b_i)\\
  &=\sum_{j=1}^{k+1} (a_j-c_j)\prod_{i<j}(a_i-b_i)\prod_{i>j}(c_i-b_i),
\end{align*}
so the result also holds for $m=k+1$. This completes the proof.
\end{proof}

The next lemma establishes a useful relation between $A_{\vec{p}}$ and $A_{\vec{p},q}$ classes that we shall need in the sequel. Given $\vec{p}=(p_1,\dots,p_m)$ with $1/p=\sum_{i=1}^{m}1/{p_i}$ and  $1\leq p_i\leq \infty$ for every $i$, we say that $\vec{w}=(w_1,\dots,w_m)\in A_{\vec{p}}$ if
\[\sup_B \left(\frac{1}{|B|}\int_B \prod_{i=1}^m w_i^{p/p_i}\right)^{1/p}\prod_{i\in\mathcal{I}_1}\left\|w_i^{-1}\mathcal{X}_B\right\|_\infty\prod_{i\in \mathcal{I}_2}\left(\frac{1}{|B|}\int_B w_i^{1-p_i'}\right)^{1/p_i'}<\infty,\]
and this supremum is denoted by $[\vec{w}]_{A_{\vec{p}}}$. 

On the other hand, given $0<q<\infty$ and $\vec{p}$ as above, we say that $\vec{w}=(w_1,\dots,w_m)\in A_{\vec{p},q}$ if
\[\sup_B \left(\frac{1}{|B|}\int_B \prod_{i=1}^m w_i^q\right)^{1/q}\prod_{i\in\mathcal{I}_1}\left\|w_i^{-1}\mathcal{X}_B\right\|_\infty\prod_{i\in \mathcal{I}_2}\left(\frac{1}{|B|}\int_B w_i^{-p_i'}\right)^{1/p_i'}<\infty,\]
and this supremum is denoted by $[\vec{w}]_{A_{\vec{p},q}}$. When $q=\infty$ this class corresponds to $A_{\vec{p},\infty}$ given in \eqref{eq: clase Ap,infinito}.

\begin{lema}\label{lema: relacion entre Ap y Ap,q}
    Let $\vec{p}=(p_1,\dots,p_m)$ be a vector of exponents, $0<q<\infty$ and $\vec{w}=(w_1,\dots,w_m)$. Assume that
    \begin{equation}\label{eq: lema: relacion entre Ap y Ap,q - eq1}
        \frac{1}{p_i}+\frac{1}{mq}-\frac{1}{mp}>0
    \end{equation}
    for every $1\leq i\leq m$.
    We define $\lambda_m=1/(mp)'+1/(mq)$,
    \[\ell_i=\left\{\begin{array}{ccl}
        1 & \textrm{ if } & i\in \mathcal{I}_1  \\
        (\lambda_m p_i')' & \textrm{ if } &  i\in \mathcal{I}_2, 
    \end{array}\right.\]
    $\ell$ such that $1/\ell=\sum_{i=1}^m 1/\ell_i$ and $\vec{z}=(z_1,\dots,z_m)$ where $z_i=w_i^{q\ell_i/\ell}$, for each $i$. Then $\vec{w}\in A_{\vec{p},q}$ if and only if $\vec{z}\in A_{\vec{\ell}}$.
\end{lema}

\begin{proof}
We first notice that condition \eqref{eq: lema: relacion entre Ap y Ap,q - eq1} guarantees that $\ell_i>1$ for $i\in\mathcal{I}_2$. It is immediate from the definition that
\[\prod_{i=1}^m z_i^{\ell/\ell_i}=\prod_{i=1}^m w_i^q.\]
Observe that
\[\frac{1}{\ell}=\sum_{i=1}^m\frac{1}{\ell_i}=\sum_{i\in \mathcal{I}_1}1+\sum_{i\in\mathcal{I}_2}\frac{1}{(\lambda_mp_i')'}=m-\frac{1}{\lambda_m}\left(m-\frac{1}{p}\right)=\frac{1}{q\lambda_m}.\]
Also notice that for $i\in\mathcal{I}_2$ 
\[z_i^{1-\ell_i'}=w_i^{q\ell_i(1-\ell_i')/\ell}=w_i^{-\ell_i'/\lambda_m}=w_i^{-p_i'}\]
and
\[\frac{1}{\ell_i'}=\frac{1}{\lambda_mp_i'}=\frac{q}{\ell p_i'}.\]
These identities imply that
\[\left(\frac{1}{|B|}\int_B \prod_{i=1}^m z_i^{\ell/\ell_i}\right)^{1/\ell}\prod_{i\in\mathcal{I}_1}\left\|z_i^{-1}\mathcal{X}_B\right\|_\infty\prod_{i\in \mathcal{I}_2}\left(\frac{1}{|B|}\int_B z_i^{1-\ell_i'}\right)^{1/\ell_i'}\]
can be rewritten as
\[ \left[\left(\frac{1}{|B|}\int_B \prod_{i=1}^m w_i^q\right)^{1/q}\prod_{i\in\mathcal{I}_1}\left\|w_i^{-1}\mathcal{X}_B\right\|_\infty\prod_{i\in \mathcal{I}_2}\left(\frac{1}{|B|}\int_B w_i^{-p_i'}\right)^{1/p_i'}\right]^{q/\ell},\]
from where the equivalence follows.
\end{proof}

\begin{obs}
When $m=1$ condition \eqref{eq: lema: relacion entre Ap y Ap,q - eq1} trivially holds, and we get that $w\in A_{p,q}$ if and only if $w^q\in A_{1+q/p'}$, a well-known relation between $A_p$ and $A_{p,q}$ classes.
\end{obs}

\section{Proof of Theorem~\ref{teo: acotacion Lp Lipschitz para T_alpha,b (suma)}}\label{seccion: conmutador suma}

We devote this section to prove Theorem~\ref{teo: acotacion Lp Lipschitz para T_alpha,b (suma)}. We shall begin with an auxiliary lemma that will be useful for this purpose.

\begin{lema}\label{lema: acotacion local de Ialpha,m}
Let $0<\alpha<mn$, $0<\delta<mn-\alpha$, $\tilde\alpha=\alpha+\delta$ and $\vec{p}$ a vector of exponents that satisfies $p>n/\tilde\alpha$. Let $\tilde\delta\leq \delta$ and  $(w,\vec{v})$ be a pair of weights belonging  to the class $\mathbb{H}_m(\vec{p},\tilde\alpha,\tilde\delta)$ such that $v_i^{-p_i'}\in\mathrm{RH}_m$ for every $i\in\mathcal{I}_2$. Then there exists a positive constant $C$ such that for every ball $B$ and every $\vec{f}$ such that $f_iv_i\in L^{p_i}$, $1\leq i\leq m$, we have that
\[\int_B |I_{\tilde\alpha,m}\vec{g}(x)|\,dx\leq C\frac{|B|^{1+\tilde\delta/n}}{\|w\mathcal{X}_B\|_\infty}\prod_{i=1}^m\|f_iv_i\|_{p_i},\]
where $\vec{g}=(f_1\mathcal{X}_{2B},f_2\mathcal{X}_{2B},\dots,f_m\mathcal{X}_{2B})$. 
\end{lema}

\begin{proof}
We shall follow similar lines as in the proof of Lemma 3.1 in \cite{BPR22}. We include a sketch for the sake of completeness. 

We shall split the set $\mathcal I_2$ into $\mathcal I_2^1$ and $\mathcal I_2^2$ where
		\[\mathcal{I}_2^1=\{i\in \mathcal{I}_2 : 1<p_i<\infty\}\quad \textrm{ and } \quad\mathcal{I}_2^2=\{i \in \mathcal{I}_2: p_i=\infty\}.\]
		Let $m_i=\#\mathcal{I}_i$, for $i=1,2$ and $m_2^j=\#\mathcal{I}_2^j$, also for $j=1,2$. Then $m=m_1+m_2=m_1+m_2^1+m_2^2$. Then, by denoting $\tilde B=2B$, for $x\in B$ we have that
		\begin{align*}
		|I_{\tilde\alpha,m}\vec{g}(x)|&\leq \int_{\tilde B^m }\frac{\prod_{i=1}^m|f_i(y_i)|}{(\sum_{i=1}^m|x-y_i|)^{mn-\tilde\alpha}}\,d\vec{y}\\
		&\leq \left(\prod_{i\in\mathcal{I}_2^2}\|f_iv_i\|_\infty\right)\int_{\tilde B^m}\frac{\prod_{i\in \mathcal{I}_1\cup\mathcal{I}_2^1}|f_i(y_i)|\prod_{i\in \mathcal{I}_2^2}v_i^{-1}(y_i)}{(\sum_{i=1}^m|x-y_i|)^{mn-\tilde\alpha}}\,d\vec{y}\\
		&\leq \left(\prod_{i\in\mathcal{I}_2^2}\|f_iv_i\|_\infty\right)\left(\prod_{i\in \mathcal{I}_1}\|f_i\mathcal{X}_{\tilde B}\|_1\right) \int_{\tilde B^{m_2}}\frac{\prod_{i\in \mathcal{I}_2^1}|f_i(y_i)|\prod_{i\in \mathcal{I}_2^2}v_i^{-1}(y_i)}{(\sum_{i\in \mathcal{I}_2}|x-y_i|)^{mn-\tilde\alpha}}\,d\vec{y}\\
		&=\left(\prod_{i\in\mathcal{I}_2^2}\|f_iv_i\|_\infty\right)\left(\prod_{i\in \mathcal{I}_1}\|f_i\mathcal{X}_{\tilde B}\|_1\right)I(x,B).
		\end{align*}
		Since $p>n/\tilde\alpha$ we have that
		\[\tilde\alpha>n/p=n\sum_{i=1}^m\frac{1}{p_i}=m_1n+\frac{n}{p^*},\]
		where $1/p^*=\sum_{i\in\mathcal{I}_2}1/p_i$. Then we can split $\tilde\alpha=\tilde\alpha^1+\tilde\alpha^2$, where $\tilde\alpha^1>m_1n$ and $\tilde\alpha^2>n/p^*$. Therefore
		\[mn-\tilde\alpha=m_2n-\tilde\alpha^2+m_1n-\tilde\alpha^1.\]
		Let us sort the sets $\mathcal{I}_2^1$ and $\mathcal{I}_2^2$ increasingly, so
		\[\mathcal{I}_2^1=\left\{i_1,i_2,\dots,i_{m_2^1}\right\} \quad \textrm{ and } \quad \mathcal{I}_2^2=\left\{i_{m_2^1+1},i_{m_2^1+2},\dots,i_{m_2}\right\}.\]
		We now define $\vec{g}=(g_1,\dots,g_{m_2})$, where
		\[g_j=\left\{\begin{array}{ccl}
		|f_{i_j}|&\textrm{ if }&1\leq j\leq m_2^1;\\
		v_{i_j}^{-1}&\textrm{ if }&m_2^1+1\leq j\leq m_2.
		\end{array}
		\right.
		\]
		Then we can proceed in the following way
		\begin{align*}
		I(x,B)&\lesssim \int_{\tilde B^{m_2}}\frac{\prod_{i\in \mathcal{I}_2^1}|f_i(y_i)|\prod_{i\in \mathcal{I}_2^2}v_i^{-1}(y_i){(\sum_{i\in \mathcal{I}_2}|x-y_i|)^{{\tilde\alpha^1-nm_1}}}}{(\sum_{i\in \mathcal{I}_2}|x-y_i|)^{m_2n-\tilde\alpha^2}}\,d\vec{y}\\
		&\lesssim |\tilde B|^{\tilde\alpha^1/n-m_1}\int_{\tilde B^{m_2}}\frac{\prod_{i\in \mathcal{I}_2^1}|f_i(y_i)|\prod_{i\in \mathcal{I}_2^2}v_i^{-1}(y_i)}{(\sum_{i\in \mathcal{I}_2}|x-y_i|)^{m_2n-\tilde\alpha^2}}\,d\vec{y}\\
		&=|\tilde B|^{\tilde\alpha^1/n-m_1}\int_{\tilde B^{m_2}}\frac{\prod_{j=1}^{m_2}g_j(y_{i_j})}{(\sum_{j=1}^{m_2}|x-y_{i_j}|)^{m_2n-\tilde\alpha^2}}\,d\vec{y}\\
		&\lesssim |\tilde B|^{\tilde\alpha^1/n-m_1}I_{\tilde\alpha^2,m_2}(\vec{g}\mathcal{X}_{\tilde B^{m_2}})(x).
		\end{align*}
		Next we define the vector of exponents $\vec{r}=(r_1,\dots,r_{m_2})$ in the following way
		\[r_j=\left\{\begin{array}{ccr}
		m_2p_{i_j}/(m_2-1+p_{i_j})&\textrm{ if }&1\leq j\leq m_2^1;\\
		m_2&\textrm{ if }&m_2^1+1\leq j\leq m_2.
		\end{array}
		\right.
		\]
		This definition yields
		\begin{align*}
		\frac{1}{r}&=\sum_{j=1}^{m_2}\frac{1}{r_j}=\sum_{j=1}^{m_2^1}\left(\frac{1}{m_2}+\frac{m_2-1}{m_2p_{i_j}}\right)+\sum_{j=m_2^1+1}^{m_2}\frac{1}{m_2}=\frac{m_2^1}{m_2}+\frac{m_2-1}{m_2{p^*}}+\frac{m_2^2}{m_2}=1+\frac{m_2-1}{m_2p^*}.
		\end{align*}
		
		Notice that $1/r>1/p^*$ and  also  $n/p^*<\tilde\alpha^2$ by construction. Then there exists an auxiliary number $\tilde\alpha_0$ such that $n/p^*<\tilde\alpha_0<n/r$. Indeed, if $\tilde\alpha^2<n/r$ we can directly pick $\tilde\alpha_0=\tilde\alpha^2$. Otherwise $\tilde\alpha_0<\tilde\alpha^2$. We shall first assume that $m_2\geq 2$. We set
		\[\frac{1}{q}=\frac{1}{r}-\frac{\tilde\alpha_0}{n}.\]
		Then $0<1/q<1$ since
		\[\frac{1}{r}-1=\left(1-\frac{1}{m_2}\right)\frac{1}{p^*}<\frac{1}{p^*}<\frac{\tilde\alpha_0}{n}.\]
				By using the well-known continuity property $I_{\tilde\alpha_0,m_2}\colon \prod_{j=1}^{m_2} L^{r_j}\to L^q$ with respect to the Lebesgue measure (see, for example, \cite{Moen09}) we obtain 
		\begin{align*}
		\int_B I(x,B)\,dx&\lesssim |\tilde B|^{\tilde\alpha^1/n-m_1+(\tilde\alpha^2-\tilde\alpha_0)/n}\left(\int_B|I_{\tilde\alpha_0,m_2}(\vec{g}\mathcal{X}_{\tilde B^{m_2}})(x)|^q\,dx\right)^{1/q}|B|^{1/q'}\\
		&\lesssim |\tilde B|^{(\tilde\alpha-\tilde\alpha_0)/n-m_1+1/q'}\left(\int_{\mathbb{R}^n}|I_{\tilde\alpha_0,m_2}(\vec{g}\mathcal{X}_{\tilde B^{m_2}})(x)|^q\,dx\right)^{1/q}\\
		&\lesssim |\tilde B|^{(\tilde\alpha-\tilde\alpha_0)/n-m_1+1/q'}\prod_{j=1}^{m_2}\|g_j\mathcal{X}_{\tilde B}\|_{r_j}.
		\end{align*}
		Observe that $r_j<p_{i_j}$ for every $1\leq j\leq m_2^1$. Since $v_i^{-p_i'}\in \mathrm{RH}_m\subseteq \mathrm{RH}_{m_2}$ for every $i\in\mathcal{I}_2$, by applying H\"{o}lder inequality and then the reverse Hölder condition on these weights we get
		\begin{align*}
		\prod_{j=1}^{m_2}\|g_j\mathcal{X}_{\tilde B}\|_{r_j}&=\prod_{i\in\mathcal{I}_2^1}\left(\int_{\tilde B}|f_i|^{r_i}v_i^{r_i}v_i^{-r_i}\right)^{1/r_i}\prod_{i\in \mathcal{I}_2^2}\left(\int_{\tilde B}v_i^{-m_2}\right)^{1/m_2}\\
		&\leq \prod_{i\in\mathcal{I}_2^1}\|f_iv_i\|_{p_i}\left(\int_{\tilde B} v_i^{-m_2p_i'}\right)^{1/(m_2p_i')}\prod_{i\in \mathcal{I}_2^2}\left(\int_{\tilde B}v_i^{-m_2}\right)^{1/m_2}\\
		&\leq |\tilde B|^{m_2^1/m_2-1/(m_2p^*)+m_2^2/m_2}\prod_{i\in\mathcal{I}_2^1}\left[v_i^{-p_i'}\right]_{\mathrm{RH}_{m_2}}\|f_iv_i\|_{p_i}\left(\frac{1}{|\tilde B|}\int_{\tilde B} v_i^{-p_i'}\right)^{1/p_i'}\\
		&\quad \times \prod_{i\in \mathcal{I}_2^2}\left[v_i^{-1}\right]_{\mathrm{RH}_{m_2}}\left(\frac{1}{|\tilde B|}\int_{\tilde B}v_i^{-1}\right)\\
		&\lesssim |\tilde B|^{1-1/(m_2p^*)}\prod_{i\in\mathcal{I}_2^1}\|f_iv_i\|_{p_i}\left(\frac{1}{|\tilde B|}\int_{\tilde B} v_i^{-p_i'}\right)^{1/p_i'}\prod_{i\in \mathcal{I}_2^2}\left(\frac{1}{|\tilde B|}\int_{\tilde B}v_i^{-1}\right).
		\end{align*}
		By combining the estimates above with condition \eqref{eq: condicion local}, we finally arrive to
		\begin{align*}
		\int_B |I_{\tilde\alpha,m}\vec{g}(x)|\,dx&\leq \left(\prod_{i\in\mathcal{I}_2^2}\|f_iv_i\|_\infty\right)\left(\prod_{i\in \mathcal{I}_1}\|f_i\mathcal{X}_{\tilde B}\|_1\right)\int_B I(x,B)\,dx\\
		&\lesssim \left(\prod_{i\in\mathcal{I}_2^2}\|f_iv_i\|_\infty\right)\left(\prod_{i\in \mathcal{I}_1}\|f_i\mathcal{X}_{\tilde B}\|_1\right)|\tilde B|^{(\tilde\alpha-\tilde\alpha_0)/n-m_1+1/q'+1-1/(m_2p^*)}\\
		&\quad \times \prod_{i\in\mathcal{I}_2^1}\|f_iv_i\|_{p_i}\left(\frac{1}{|\tilde B|}\int_{\tilde B} v_i^{-p_i'}\right)^{1/p_i'}\prod_{i\in \mathcal{I}_2^2}\left(\frac{1}{|\tilde B|}\int_{\tilde B}v_i^{-1}\right)\\
		&\lesssim \left(\prod_{i=1}^m \|f_iv_i\|_{p_i}\right)\prod_{i\in \mathcal{I}_2}\left(\frac{1}{|\tilde B|}\int_{\tilde B} v_i^{-p_i'}\right)^{1/p_i'}\prod_{i\in\mathcal{I}_1} \left\|v_i^{-1}\mathcal{X}_{\tilde B}\right\|_\infty \\
		&\quad \times |\tilde B|^{(\tilde\alpha-\tilde\alpha_0)/n-m_1+1/q'+1-1/(m_2p^*)}\\
		&\lesssim  \|w\mathcal{X}_{\tilde B}\|_\infty^{-1}\, |\tilde B|^{\tilde\delta/n-\tilde\alpha/n+1/p+(\tilde\alpha-\tilde\alpha_0)/n-m_1+1/q'+1-1/(m_2p^*)}\left(\prod_{i=1}^m \|f_iv_i\|_{p_i}\right)\\
		&\lesssim  \|w\mathcal{X}_{B}\|_\infty^{-1}|B|^{1+\tilde\delta/n}\left(\prod_{i=1}^m \|f_iv_i\|_{p_i}\right).
		\end{align*}
		Thus we can achieve the desired estimate provided $m_2\geq 2$. We shall now consider $0\leq m_2<2$. There are only three possible cases:  
		\begin{enumerate}
			\item $m_2=0$. In this case we have $m_2^1=m_2^2=0$ and this implies $\vec{p}=(1,1,\dots,1)$. This situation is not possible, because $p>n/\tilde\alpha$.
			\item $m_2^1=0$ and $m_2^2=1$. In this case $1/p=m-1$. Condition $p>n/\tilde\alpha$ implies $\tilde\alpha>(m-1)n$. Let $i_0$ be the index such that $p_{i_0}=\infty$. By using Fubini's theorem, we can proceed in the following way
			\[\int_B \int_{\tilde B^m}\frac{\prod_{i=1}^m|f_i(y_i)|}{(\sum_{i=1}^m|x-y_i|)^{mn-\tilde\alpha}}\,d\vec{y}\,dx=\int_{\tilde B^m}\prod_{i=1}^m|f_i(y_i)|\left(\int_B \left(\sum_{i=1}^m|x-y_i|\right)^{\tilde\alpha-mn}\,dx\right)\,d\vec{y}.\]
			Since
			\begin{align*}
			\int_B \left(\sum_{i=1}^m|x-y_i|\right)^{\tilde\alpha-mn}\,dx&\lesssim \int_0^{4R}\rho^{\tilde\alpha-mn}\rho^{n-1}\,d\rho\\
			&\lesssim  |B|^{\tilde\alpha/n-m+1},
			\end{align*}
			by  \eqref{eq: condicion local}, we get
			\begin{align*}
			\int_B |I_{\tilde\alpha,m}\vec{g}(x)|\,dx&\lesssim |B|^{\tilde\alpha/n-m+2}\left(\prod_{i=1}^m\|f_iv_i\|_{p_i}\right)\left(\prod_{i\in \mathcal{I}_1}\left\|v_i^{-1}\mathcal{X}_{\tilde B}\right\|_\infty\right) \left(\frac{1}{|\tilde B|}\int_{\tilde B}v_{i_0}^{-1}\right)\\
			&\lesssim \left(\prod_{i=1}^m\|f_iv_i\|_{p_i}\right)\frac{|\tilde B|^{\tilde\alpha/n-m+2+\tilde\delta/n-\tilde\alpha/n+1/p}}{\|w\mathcal{X}_{\tilde B}\|_\infty}\\
			&\lesssim \left(\prod_{i=1}^m\|f_iv_i\|_{p_i}\right)\frac{|B|^{1+\tilde\delta/n}}{\|w\mathcal{X}_{ B}\|_\infty}.
			\end{align*}
			\item $m_2^1=1$ and $m_2^2=0$. If $i_0$ denotes the index for which $1<p_{i_0}<\infty$, the condition $p>n/\tilde\alpha$ implies that
			\[\frac{\tilde\alpha}{n}>\frac{1}{p}=m-1+\frac{1}{p_{i_0}},\]
			and thus $\tilde\alpha>(m-1)n$. We repeat the estimate given in the previous case. This yields
			\begin{align*}
			\int_B |I_1\vec{f}(x)|\,dx&\lesssim |B|^{\tilde\alpha/n-m+1+1/p_{i_0}'}\left(\prod_{i=1}^m\|f_iv_i\|_{p_i}\right)\left(\prod_{i\in \mathcal{I}_1}\left\|v_i^{-1}\mathcal{X}_{\tilde B}\right\|_\infty\right)\left(\frac{1}{|\tilde B|}\int_{\tilde B}v_{i_0}^{-p_{i_0}'}\right)^{1/p_{i_0}'}\\
			&\lesssim \left(\prod_{i=1}^m\|f_iv_i\|_{p_i}\right)\frac{|\tilde B|^{\tilde\alpha/n-m+1+1/p_{i_0}'+\tilde\delta/n-\tilde\alpha/n+1/p}}{\|w\mathcal{X}_{\tilde B}\|_\infty}\\
			&\lesssim \left(\prod_{i=1}^m\|f_iv_i\|_{p_i}\right)\frac{|B|^{1+\tilde\delta/n}}{\|w\mathcal{X}_{ B}\|_\infty}.
			\end{align*}
		\end{enumerate}
		We covered all the possible cases for $m_2$ and the proof is complete.\qedhere
\end{proof}

\medskip

\begin{proof}[Proof of Theorem~\ref{teo: acotacion Lp Lipschitz para T_alpha,b (suma)} ]
It will be enough to prove that
\begin{equation}\label{eq: teo: acotacion Lp Lipschitz para T_alpha,b (suma) - eq1}
\frac{\|w\mathcal{X}_B\|_\infty}{|B|^{1+\tilde\delta/n}}\int_B |T_{\alpha,b_j}^{m}\vec{f}(x)-c_j|\,dx\leq C\prod_{i=1}^m\|f_iv_i\|_{p_i},
\end{equation}
for some positive constant $c_j$ and every ball $B$, for each $j$ and with $C$ independent of $B$ and $j$. Indeed, if \eqref{eq: teo: acotacion Lp Lipschitz para T_alpha,b (suma) - eq1} holds we take $c=\sum_{j=1}^m c_j$ and therefore
\begin{align*}
\frac{\|w\mathcal{X}_B\|_\infty}{|B|^{1+\tilde\delta/n}}\int_B |T_{\alpha,\mathbf{b}}^{m}\vec{f}(x)-c|\,dx&\leq \sum_{j=1}^m \frac{\|w\mathcal{X}_B\|_\infty}{|B|^{1+\tilde\delta/n}}\int_B |T_{\alpha,b_j}^{m}\vec{f}(x)-c_j|\,dx\\
&\leq Cm\prod_{i=1}^m\|f_iv_i\|_{p_i}
\end{align*}
and the proof would be complete. Then we shall proceed to prove \eqref{eq: teo: acotacion Lp Lipschitz para T_alpha,b (suma) - eq1}.

Fix $1\leq j\leq m$ and a ball $B=B(x_B, R)$. We decompose $\vec{f}=(f_1,f_2,\dots,f_m)$ as $\vec{f}=\vec{f}_1+\vec{f}_2$, where $\vec{f}_1=(f_1\mathcal{X}_{2B},f_2\mathcal{X}_{2B},\dots,f_m\mathcal{X}_{2B})$. We take 
\[c_j=\left(T_{\alpha,b_j}^{m}\vec{f}_2\right)_B=\frac{1}{|B|}\int_B T_{\alpha,b_j}^{m}\vec{f}_2(z)\,dz.\]
We also notice that
\begin{equation}\label{eq: teo: acotacion Lp Lipschitz para T_alpha,b (suma) - eq2}
\frac{1}{|B|}\int_B T_{\alpha,b_j}^{m}\vec{f}_2(x)\,dx=\sum_{\sigma\in S_m,\sigma\neq \bf{1}}\frac{1}{|B|}\int_B\int_{(\mathbf{2B})^\sigma} (b_j(x)-b_j(y_j))K_\alpha(x,\vec{y})\prod_{i=1}^mf_i(y_i)\,d\vec{y}\,dx.
\end{equation}

In order to prove \eqref{eq: teo: acotacion Lp Lipschitz para T_alpha,b (suma) - eq1} we write
\begin{align*}
    \frac{\|w\mathcal{X}_B\|_\infty}{|B|^{1+\tilde\delta/n}}\int_B |T_{\alpha,b_j}^{m}\vec{f}(x)-c_j|\,dx&\leq\frac{\|w\mathcal{X}_B\|_\infty}{|B|^{1+\tilde\delta/n}}\left(\int_B |T_{\alpha,b_j}^{m}\vec{f}_1(x)|\,dx+\frac{1}{|B|}\int_B|T_{\alpha,b_j}^{m}\vec{f}_2(x)-c_j|\,dx\right)\\
    &=\frac{\|w\mathcal{X}_B\|_\infty}{|B|^{1+\tilde\delta/n}}\left(I+\frac{1}{|B|}II\right).
\end{align*}
Let us first estimate $I$. By applying Lemma~\ref{lema: acotacion local de Ialpha,m} we get 
\begin{align*}
I=\int_B |T_{\alpha,b_j}^{m}\vec{f}_1(x)\,dx|&\leq\int_B\int_{(2B)^m}|b_j(x)-b_j(y_j)|\,|K_\alpha(x,\vec{y})|\prod_{i=1}^m|f_i(y_i)|\,d\vec{y}\,dx\\
&\lesssim\|b_j\|_{\Lambda(\delta)}\int_B\int_{(2B)^m}|K_{\tilde\alpha}(x,\vec{y})|\prod_{i=1}^m|f_i(y_i)|\,d\vec{y}\,dx\\
&\lesssim\|\mathbf{b}\|_{(\Lambda(\delta))^m}\int_B |I_{\tilde\alpha,m}\vec{f}_1(x)|\,dx\\
&=\|\mathbf{b}\|_{(\Lambda(\delta))^m}\frac{|B|^{1+\tilde\delta/n}}{\|w\mathcal{X}_B\|_\infty}\prod_{i=1}^m\|f_iv_i\|_{p_i}.
\end{align*}
Consequently,
\[\frac{\|w\mathcal{X}_B\|_\infty}{|B|^{1+\tilde\delta/n}}\,\,I\lesssim \|\mathbf{b}\|_{(\Lambda(\delta))^m}\prod_{i=1}^m\|f_iv_i\|_{p_i}.\]

We now turn our attention to $II$. By \eqref{eq: teo: acotacion Lp Lipschitz para T_alpha,b (suma) - eq2} we can write
\begin{align*}
II&\leq\sum_{\sigma\in S_m,\sigma\neq \bf{1}}\int_B\int_B\int_{(\mathbf{2B})^\sigma}\left|(b_j(x)-b_j(y_j))K_\alpha(x,\vec{y})-(b_j(z)-b_j(y_j))K_\alpha(z,\vec{y})\right|\\
&\qquad \times\prod_{i=1}^m|f_i(y_i)|\,d\vec{y}\,dx\,dz\\
&\leq \sum_{\sigma\in S_m,\sigma\neq \bf{1}}\int_B\int_B\int_{(\mathbf{2B})^\sigma}\left|(b_j(x)-b_j(y_j))(K_\alpha(x,\vec{y})-K_\alpha(z,\vec{y}))\right|\prod_{i=1}^m|f_i(y_i)|\,d\vec{y}\,dx\,dz\,\,\\
&\quad + \sum_{\sigma\in S_m,\sigma\neq \bf{1}}\int_B\int_B\int_{(\mathbf{2B})^\sigma}\left|(b_j(x)-b_j(z))K_\alpha(z,\vec{y})\right|\prod_{i=1}^m|f_i(y_i)|\,d\vec{y}\,dx\,dz\\
&= \sum_{\sigma\in S_m,\sigma\neq \bf{1}} (I_1^\sigma+I_2^\sigma).
\end{align*} 
We shall estimate each sum separately. Fix $\sigma\in S_m,\sigma\neq \bf{1}$. We start with $I_1^\sigma$. Since we are assuming $\sigma\neq \bf{1}$, condition \eqref{eq: condicion de suavidad} implies that
\begin{align*}
  |K_\alpha(x,\vec{y})-K_\alpha(z,\vec{y})|&\lesssim \frac{|x-z|^\gamma}{(\sum_{i=1}^m |x-y_i|)^{mn-\alpha+\gamma}}\\
  &\lesssim \frac{|B|^{\gamma/n}}{(\sum_{i=1}^m |x-y_i|)^{mn-\alpha+\gamma}}.
\end{align*}
Therefore we have that
\begin{align*}
I_1^\sigma&\lesssim\|b_j\|_{\Lambda(\delta)}|B|^{\gamma/n}\int_B\int_B \int_{(\mathbf{2B})^\sigma}\frac{|x-y_j|^\delta\prod_{i=1}^m|f_i(y_i)|}{(\sum_{i=1}^m|x-y_i|)^{mn-\alpha+\gamma}}\,d\vec{y}\,dx\,dz\\
&\lesssim\|b_j\|_{\Lambda(\delta)} |B|^{1+\gamma/n}\int_B\int_{(\mathbf{2B})^\sigma}\frac{\prod_{i=1}^m|f_i(y_i)|}{(\sum_{i=1}^m|x-y_i|)^{mn-\tilde\alpha+\delta+\gamma-\delta}}\,d\vec{y}\,dx\\
&\lesssim \|\mathbf{b}\|_{(\Lambda(\delta))^m}|B|^{1+\delta/n}\int_B\int_{(\mathbf{2B})^\sigma}\frac{\prod_{i=1}^m|f_i(y_i)|}{(\sum_{i=1}^m|x-y_i|)^{mn-\tilde\alpha+\delta}}\,d\vec{y}\,dx\\
&=\|\mathbf{b}\|_{(\Lambda(\delta))^m}|B|^{1+\delta/n}\int_B J_1(x,\sigma)\,dx.
\end{align*} 
\color{black}
By\refstepcounter{BPR}\label{pag: estimacion de J_1(x,sigma)} separating the factors in $J_1$ and applying Hölder inequality we arrive to
\begin{align*}
J_1(x,\sigma)&\lesssim \left(\prod_{i:\sigma_i=1} \int_{2B}\frac{|f_i(y_i)|}{|2B|^{1-\tilde \alpha_i/n+\delta/(mn)}}\,dy_i\right)\left(\prod_{i: \sigma_i=0} \int_{\mathbb{R}^n\backslash 2B}\frac{|f_i(y_i)|}{|x-y_i|^{n-\tilde\alpha_i+\delta/m}}\,dy_i\right)\\
&\lesssim \prod_{i=1}^m\|f_iv_i\|_{p_i}\left(\prod_{i:\sigma_i=1} \left\|\frac{v_i^{-1}\mathcal{X}_{ 2B}}{|2B|^{1-\tilde\alpha_i/n+\delta/(mn)}}\right\|_{p_i'}\right)\left(\prod_{i:\sigma_i=0} \left\|\frac{v_i^{-1}\mathcal{X}_{\mathbb{R}^n\backslash 2B}}{|x-\cdot|^{n-\tilde\alpha_i+\delta/m}}\right\|_{p_i'}\right)\\
&=\left(\prod_{i=1}^m\|f_iv_i\|_{p_i}\right)|2B|^{-\sum_{i:\sigma_i=1}(1-\tilde\alpha_i/n+\delta/(mn))}\prod_{i: \sigma_i=1} \|v_i^{-1}\mathcal{X}_{2B}\|_{p_i'}\\
&\qquad \times \prod_{i:\sigma_i=0}\left\|\frac{v_i^{-1}\mathcal{X}_{\mathbb{R}^n\backslash 2B}}{|x-\cdot|^{n-\tilde\alpha_i+\delta/m}}\right\|_{p_i'}\\
&\lesssim \left(\prod_{i=1}^m\|f_iv_i\|_{p_i}\right) \frac{|2B|^{(\tilde\delta-\delta)/n}}{\|w\mathcal{X}_{2B}\|_\infty},
\end{align*}
by virtue of condition \eqref{eq: condicion mezclada para sigma}. Thus
\begin{equation}\label{eq: teo: acotacion Lp Lipschitz para T_alpha,b (suma) - eq3}
I_1^\sigma\lesssim\|\mathbf{b}\|_{(\Lambda(\delta))^m}\left(\prod_{i=1}^m\|f_iv_i\|_{p_i}\right)\frac{|B|^{2+\tilde\delta/n}}{\|w\mathcal{X}_B\|_\infty}.
\end{equation}
We now proceed to estimate $I_2^\sigma$. We have that 
\begin{align*}
I_2^\sigma&\lesssim \|b_j\|_{\Lambda(\delta)}|B|^{\delta/n}\int_B\int_B\int_{(\mathbf{2B})^\sigma}\frac{\prod_{i=1}^m|f_i(y_i)|}{(\sum_{i=1}^m |z-y_i|)^{mn-\alpha}}\,d\vec{y}\,dx\,dz\\
&\lesssim \|\mathbf{b}\|_{(\Lambda(\delta))^m}|B|^{2+\delta/n}\left(\prod_{i:\sigma_i=1} \int_{2B} \frac{|f_i(y_i)|}{|2B|^{1-\alpha_i/n}}\,dy_i\right)\left(\prod_{i:\sigma_i=0} \int_{\mathbb{R}^n\backslash 2B}\frac{|f_i(y_i)|}{|x_B-y_i|^{n-\alpha_i}}\,dy_i\right).
\end{align*}
Since $\tilde\alpha_i=\alpha_i+\delta_i/m$ for each $i$, by applying Hölder inequality we can write
\begin{align*}
\prod_{i:\sigma_i=1} \int_{2B} \frac{|f_i(y_i)|}{|2B|^{1-\alpha_i/n}}\,dy_i&\lesssim \prod_{i:\sigma_i=1}\frac{\|f_iv_i\|_{p_i}}{|2B|^{\sum_{i:\sigma_i=1}(1/p_i-\alpha_i/n)}}\left\|v_i^{-1}\mathcal{X}_{2B}\right\|_{p_i'}\\
&= \prod_{i:\sigma_i=1}\frac{\|f_iv_i\|_{p_i}}{|2B|^{\sum_{i:\sigma_i=1}(1/p_i-\tilde\alpha_i/n+\delta/(mn))}}\left\|v_i^{-1}\mathcal{X}_{2B}\right\|_{p_i'}
\end{align*}
and
\[\prod_{i:\sigma_i=0} \int_{\mathbb{R}^n\backslash 2B}\frac{|f_i(y_i)|}{|x_B-y_i|^{n-\alpha_i}}\,dy_i\lesssim \|f_iv_i\|_{p_i}\left\|\frac{v_i^{-1}\mathcal{X}_{\mathbb{R}^n\backslash 2B}}{(|B|^{1/n}+|x_B-\cdot|)^{n-\tilde\alpha_i+\delta/m}}\right\|_{p_i'}.\]
By combining these estimates and using condition \eqref{eq: condicion mezclada para sigma} we arrive to
\begin{equation}\label{eq: teo: acotacion Lp Lipschitz para T_alpha,b (suma) - eq4}
I_2^\sigma\lesssim \|\mathbf{b}\|_{(\Lambda(\delta))^m}|B|^{2+\delta/n}\prod_{i=1}^m \|f_iv_i\|_{p_i}\frac{|2B|^{(\tilde\delta-\delta)/n}}{\|w\mathcal{X}_B\|_\infty}=\|\mathbf{b}\|_{(\Lambda(\delta))^m}\frac{|B|^{2+\tilde\delta/n}}{\|w\mathcal{X}_B\|_\infty}\prod_{i=1}^m \|f_iv_i\|_{p_i}.
\end{equation}
Therefore, by applying the estimates obtained in \eqref{eq: teo: acotacion Lp Lipschitz para T_alpha,b (suma) - eq3} and \eqref{eq: teo: acotacion Lp Lipschitz para T_alpha,b (suma) - eq4} we conclude that
\[\frac{1}{|B|}II\lesssim \|\mathbf{b}\|_{(\Lambda(\delta))^m}\frac{|B|^{1+\tilde\delta/n}}{\|w\mathcal{X}_B\|_\infty}\prod_{i=1}^m \|f_iv_i\|_{p_i}.\]
This completes the proof of \eqref{eq: teo: acotacion Lp Lipschitz para T_alpha,b (suma) - eq1} and we are done.
\end{proof}

\section{Proof of Theorem~\ref{teo: acotacion Lp Lipschitz para T_alpha,b (producto)}}\label{seccion: conmutador producto}

We devote this section to prove Theorem~\ref{teo: acotacion Lp Lipschitz para T_alpha,b (producto)}. We shall first establish an auxiliary lemma, which is essentially the boundedness given in Lemma~\ref{lema: acotacion local de Ialpha,m} with different parameters. The proof can be achieved by following the same steps and we shall omit it.

\begin{lema}\label{lema: acotacion local de Ialpha,m (producto)}
Let $0<\alpha<mn$, $0<\delta<(n-\alpha)/m$, $\tilde\alpha=\alpha+m\delta$ and $\vec{p}$ a vector of exponents that satisfies $p>n/\tilde\alpha$. Let $\tilde\delta\leq \delta$ and  $(w,\vec{v})$ be a pair of weights belonging  to the class $\mathbb{H}_m(\vec{p},\tilde\alpha,\tilde\delta)$ such that $v_i^{-p_i'}\in\mathrm{RH}_m$ for every $i\in\mathcal{I}_2$. Then there exists a positive constant $C$ such that for every ball $B$ and every $\vec{f}$ such that $f_iv_i\in L^{p_i}$, $1\leq i\leq m$, we have that
\[\int_B |I_{\tilde\alpha,m}\vec{g}(x)|\,dx\leq C\frac{|B|^{1+\tilde\delta/n}}{\|w\mathcal{X}_B\|_\infty}\prod_{i=1}^m\|f_iv_i\|_{p_i},\]
where $\vec{g}=(f_1\mathcal{X}_{2B},f_2\mathcal{X}_{2B},\dots,f_m\mathcal{X}_{2B})$. 
\end{lema}

\begin{proof}[Proof of Theorem~\ref{teo: acotacion Lp Lipschitz para T_alpha,b (producto)}]
It will be enough to prove that
\begin{equation}\label{eq: teo: acotacion Lp Lipschitz para T_alpha,b (producto) - eq1}
\frac{\|w\mathcal{X}_B\|_\infty}{|B|^{1+\tilde\delta/n}}\int_B |\mathcal{T}_{\alpha,\mathbf{b}}^{m}\vec{f}(x)-c|\,dx\leq C\prod_{i=1}^m\|f_iv_i\|_{p_i},
\end{equation}
for some constant $c$ and every ball $B$, with $C$ independent of $B$ and $\vec{f}$. 

Fix a ball $B=B(x_B, R)$. By proceeding as in the proof of Theorem~\ref{teo: acotacion Lp Lipschitz para T_alpha,b (suma)}, we split  $\vec{f}=\vec{f}_1+\vec{f}_2$, where $\vec{f}_1=(f_1\mathcal{X}_{2B},f_2\mathcal{X}_{2B},\dots,f_m\mathcal{X}_{2B})$. We take 
\[c=\left(\mathcal{T}_{\alpha,\mathbf{b}}^{m}\vec{f}_2\right)_B=\frac{1}{|B|}\int_B \mathcal{T}_{\alpha,\mathbf{b}}^{m}\vec{f}_2(z)\,dz.\]
By Proposition~\ref{propo: representacion de conmutador producto}, for $z\in B$ we have that
\begin{equation}\label{eq: teo: acotacion Lp Lipschitz para T_alpha,b (producto) - eq2}
 \mathcal{T}_{\alpha,\mathbf{b}}^{m}\vec{f}_2(z)=\sum_{\sigma\in S_m,\sigma\neq \bf{1}}\int_{(\mathbf{2B})^\sigma} K_\alpha(z,\vec{y})\prod_{i=1}^m(b_i(z)-b_i(y_i))f_i(y_i)\,d\vec{y}.
\end{equation}

Thus
\begin{align*}
    \frac{\|w\mathcal{X}_B\|_\infty}{|B|^{1+\tilde\delta/n}}\int_B |\mathcal{T}_{\alpha,\mathbf{b}}^{m}\vec{f}(x)-c|\,dx&\leq\frac{\|w\mathcal{X}_B\|_\infty}{|B|^{1+\tilde\delta/n}}\int_B |\mathcal{T}_{\alpha,\mathbf{b}}^{m}\vec{f}_1(x)|\,dx\\
    &\qquad +\frac{\|w\mathcal{X}_B\|_\infty}{|B|^{2+\tilde\delta/n}}\int_B\int_B|\mathcal{T}_{\alpha,\mathbf{b}}^{m}\vec{f}_2(x)-\mathcal{T}_{\alpha,\mathbf{b}}^{m}\vec{f}_2(z)|\,dz\,dx\\
    &=\frac{\|w\mathcal{X}_B\|_\infty}{|B|^{1+\tilde\delta/n}}\left(I+II\right).
\end{align*}
Let us first estimate $I$. By applying Proposition~\ref{propo: representacion de conmutador producto}, \eqref{eq: condicion de tamaño} and Lemma~\ref{lema: acotacion local de Ialpha,m (producto)} we get 
\begin{align*}
I=\int_B |\mathcal{T}_{\alpha,\mathbf{b}}^{m}\vec{f}_1(x)|\,dx&\leq\int_B\int_{(2B)^m}\,|K_\alpha(x,\vec{y})|\prod_{i=1}^m|b_i(x)-b_i(y_i)|\,|f_i(y_i)|\,d\vec{y}\,dx\\
&\leq C\prod_{i=1}^m\|b_i\|_{\Lambda(\delta)}\int_B\int_{(2B)^m}|K_{\tilde\alpha}(x,\vec{y})|\prod_{i=1}^m|f_i(y_i)|\,d\vec{y}\,dx\\
&\leq C\|\mathbf{b}\|_{(\Lambda(\delta))^m}^m\int_B |I_{\tilde\alpha,m}\vec{f}_1(x)|\,dx\\
&\leq C\|\mathbf{b}\|_{(\Lambda(\delta))^m}^m\frac{|B|^{1+\tilde\delta/n}}{\|w\mathcal{X}_B\|_\infty}\prod_{i=1}^m\|f_iv_i\|_{p_i}.
\end{align*}
Consequently,
\[\frac{\|w\mathcal{X}_B\|_\infty}{|B|^{1+\tilde\delta/n}}\,\,I\leq C \|\mathbf{b}\|_{(\Lambda(\delta))^m}^m\prod_{i=1}^m\|f_iv_i\|_{p_i}.\]

We continue with the estimate of $II$. We shall see that
\begin{equation}\label{eq: teo: acotacion Lp Lipschitz para T_alpha,b (producto) - eq3}
|\mathcal{T}_{\alpha,\mathbf{b}}^{m}\vec{f}_2(x)-\mathcal{T}_{\alpha,\mathbf{b}}^{m}\vec{f}_2(z)|\leq C\|\mathbf{b}\|_{(\Lambda(\delta))^m}^m\frac{|B|^{\tilde\delta/n}}{\|w\mathcal{X}_B\|_\infty}\prod_{i=1}^m\|f_iv_i\|_{p_i},
\end{equation}
for every $x,z\in B$. This would imply that $II\leq C\prod_{i=1}^m\|f_iv_i\|_{p_i}$.

By \eqref{eq: teo: acotacion Lp Lipschitz para T_alpha,b (producto) - eq2}, for $x\in B$ we can write
\begin{align*}
    |\mathcal{T}_{\alpha,\mathbf{b}}^{m}&\vec{f}_2(x)-\mathcal{T}_{\alpha,\mathbf{b}}^{m}\vec{f}_2(z)|\\
    &\leq \sum_{\sigma\in S_m,\sigma\neq\mathbf{1}}\int_{(\mathbf{2B})^\sigma}\left|K_\alpha(x,\vec{y})\prod_{i=1}^m(b_i(x)-b_i(y_i))- K_\alpha(z,\vec{y})\prod_{i=1}^m(b_i(z)-b_i(y_i))\right|\\
    &\qquad \times \prod_{i=1}^m|f_i(y_i)|\,d\vec{y}\\
    &\leq \sum_{\sigma\in S_m,\sigma\neq\mathbf{1}}\int_{(\mathbf{2B})^\sigma} |K_{\alpha}(x,\vec{y})-K_{\alpha}(z,\vec{y})|\prod_{i=1}^m|b_i(x)-b_i(y_i)|\,|f_i(y_i)|\,d\vec{y}\\
    &\qquad + \sum_{\sigma\in S_m,\sigma\neq\mathbf{1}}\int_{(\mathbf{2B})^\sigma} |K_{\alpha}(z,\vec{y})|\,\left|\prod_{i=1}^m(b_i(x)-b_i(y_i))-\prod_{i=1}^m(b_i(z)-b_i(y_i))\right|\\
    &\qquad \quad\times \prod_{i=1}^m|f_i(y_i)|\,d\vec{y}\\
    &=\sum_{\sigma\in S_m,\sigma\neq\mathbf{1}} (I_1^{\sigma}+I_2^{\sigma}).
\end{align*}
Fix $\sigma\in S_m,\sigma\neq \bf{1}$. Let us first estimate $I_1^\sigma$. By applying condition \eqref{eq: condicion de suavidad} we have that

\begin{align*}
  |K_\alpha(x,\vec{y})-K_\alpha(z,\vec{y})|&\lesssim \frac{|x-z|^\gamma}{(\sum_{i=1}^m |x-y_i|)^{mn-\alpha+\gamma}}\\
  &\lesssim\frac{|B|^{\gamma/n}}{(\sum_{i=1}^m |x_B-y_i|)^{mn-\alpha+\gamma}}.
\end{align*}
Therefore we have that
\begin{align*}
I_1^\sigma&\lesssim |B|^{\gamma/n}\prod_{i=1}^m\|b_i\|_{\Lambda(\delta)} \int_{(\mathbf{2B})^\sigma}\frac{(\sum_{i=1}^m|x_B-y_i|)^{m\delta}\prod_{i=1}^m|f_i(y_i)|}{(\sum_{i=1}^m|x_B-y_i|)^{mn-\alpha+\gamma}}\,d\vec{y}\\
&\lesssim\|\mathbf{b}\|_{(\Lambda(\delta))^m}^m|B|^{\gamma/n} \int_{(\mathbf{2B})^\sigma}\frac{\prod_{i=1}^m|f_i(y_i)|}{(\sum_{i=1}^m|x_B-y_i|)^{mn-\tilde\alpha+\delta+\gamma-\delta}}\,d\vec{y}\\
&\lesssim\|\mathbf{b}\|_{(\Lambda(\delta))^m}^m|B|^{\delta/n}\int_{(\mathbf{2B})^\sigma}\frac{\prod_{i=1}^m|f_i(y_i)|}{(\sum_{i=1}^m|x_B-y_i|)^{mn-\tilde\alpha+\delta}}\,d\vec{y},
\end{align*}
since $\gamma>\delta$ and $\vec{y}\in (\mathbf{2B})^{\sigma}$ implies that $|x_B-y_j|\geq C |B|^{1/n}$ for at least one index $1\leq j\leq m$. From this expression we can continue as in the estimate performed in page~\pageref{pag: estimacion de J_1(x,sigma)} in order to obtain
\[I_1^\sigma \lesssim \|\mathbf{b}\|_{(\Lambda(\delta))^m}^m\frac{|B|^{\tilde\delta/n}}{\|w\mathcal{X}_B\|_\infty}\prod_{i=1}^m\|f_iv_i\|_{p_i}.\]
\color{black}
Next we proceed to estimate $I_2^\sigma$. Fix $\sigma\in S_m, \sigma\neq\mathbf{1}$. By applying Lemma~\ref{lema: estimacion diferencia de productos} we have that
\begin{align*}
  \left|\prod_{i=1}^m(b_i(x)-b_i(y_i))-\prod_{i=1}^m(b_i(z)-b_i(y_i))\right|
  &\leq \sum_{j=1}^m |b_j(x)-b_j(z)|\prod_{i>j}|b_i(z)-b_i(y_i)|\prod_{i<j}|b_i(x)-b_i(y_i)|\\
  &\lesssim \left(\prod_{i=1}^m \|b_i\|_{\Lambda(\delta)}\right)|B|^{\delta/n}\sum_{j=1}^m\prod_{i>j}|z-y_i|^{\delta}\prod_{i<j}|x-y_i|^{\delta} 
\end{align*}
Since $x$ and $z$ belong to $B$ and $\vec{y}\in (\mathbf{2B})^\sigma$, we have that
\[|x-y_i|\lesssim \sum_{j=1}^m |x_B-y_j|\]
and also
\[|z-y_i|\lesssim \sum_{j=1}^m |x_B-y_j|,\]
for each $i$, regardless $y_i$ belongs to $2B$ or $\mathbb{R}^n\backslash 2B$.
Therefore we arrive to
\[\left|\prod_{i=1}^m(b_i(x)-b_i(y_i))-\prod_{i=1}^m(b_i(z)-b_i(y_i))\right|\lesssim  \|\mathbf{b}\|_{(\Lambda(\delta))^m}^m |B|^{\delta/n}\left(\sum_{j=1}^m|x_B-y_j|\right)^{(m-1)\delta}.\]
By using this estimate we can proceed with $I_2^\sigma$ as follows
\begin{align*}
I_2^\sigma&\lesssim \|\mathbf{b}\|_{(\Lambda(\delta))^m}^m|B|^{\delta/n}\int_{(\mathbf{2B})^\sigma}\frac{\prod_{i=1}^m |f_i(y_i)|}{\left(\sum_{i=1}^m|x_B-y_i|\right)^{mn-\alpha+(1-m)\delta}}\,d\vec{y}\\
&\lesssim\|\mathbf{b}\|_{(\Lambda(\delta))^m}^m|B|^{\delta/n}\int_{(\mathbf{2B})^\sigma}\frac{\prod_{i=1}^m |f_i(y_i)|}{\left(\sum_{i=1}^m|x_B-y_i|\right)^{mn-\tilde\alpha+\delta}}\,d\vec{y}\\
&\lesssim\|\mathbf{b}\|_{(\Lambda(\delta))^m}^m|B|^{\delta/n}\prod_{i: \sigma_i=1}\int_{2B}\frac{|f_i(y_i)|}{|2B|^{1-\tilde\alpha_i/n+\delta/(mn)}}\,dy_i\prod_{i: \sigma_i=0}\int_{\mathbb{R}^n\backslash 2B}\frac{|f_i(y_i)|}{|x_B-y_i|^{n-\tilde\alpha_i+\delta/m}}\,dy_i.
\end{align*}
By applying Hölder inequality and condition \eqref{eq: condicion mezclada para sigma} we get
\begin{align*}
    I_2^\sigma&\lesssim \|\mathbf{b}\|_{(\Lambda(\delta))^m}^m|B|^{\delta/n-\theta(\sigma)}\prod_{i=1}\|f_iv_i\|_{p_i}\prod_{i: \sigma_i=1}\|v_i^{-1}\mathcal{X}_{2B}\|_{p_i'}\prod_{i: \sigma_i=0}\left\|\frac{v_i^{-1}\mathcal{X}_{\mathbb{R}^n\backslash 2B}}{|x_B-\cdot|^{n-\tilde\alpha_i+\delta/m}}\right\|_{p_i'}\\
    &\lesssim \|\mathbf{b}\|_{(\Lambda(\delta))^m}^m \frac{|B|^{\tilde\delta/n}}{\|w\mathcal{X}_B\|_\infty}\prod_{i=1}^m \|f_iv_i\|_{p_i},
\end{align*}
where $\theta(\sigma)=\sum_{i:\sigma_i=1}\left(1-\tilde\alpha_i/n+\delta/(mn)\right) $. So \eqref{eq: teo: acotacion Lp Lipschitz para T_alpha,b (producto) - eq3} holds and the proof is complete.
\end{proof}

\section{The class \texorpdfstring{$\mathbb{H}_m(\vec{p},\beta,\tilde\delta)$}{$Hm(p,\beta,\delta)$}}\label{seccion: clase de pesos}

In this section we give a complete study of the class $\mathbb{H}_m(\vec{p},\beta,\tilde\delta)$ related with the boundedness properties stated in our main results. Recall that $(w,\vec{v})$ belongs to the class $\mathbb{H}_m(\vec{p},\beta,\tilde\delta)$ if there exists a positive constant $C$ for which the inequality
\begin{equation*}
	    \frac{\|w\mathcal{X}_B\|_\infty}{|B|^{(\tilde \delta-\delta)/n}}\prod_{i=1}^m\left(\int_{\mathbb{R}^n}\frac{v_i^{-p_i'}}{(|B|^{1/n}+|x_B-y|)^{(n-\beta_i+\delta/m)p_i'}}\,dy\right)^{1/p_i'}\leq C
	\end{equation*}
holds for every ball $B=B(x_B, R)$.

We begin with a characterizaction of this class of weights in terms of the global condition \eqref{eq: condicion global}. The proof follows similar lines as Lemma 2.1 in \cite{BPR22} and we shall omit it. We recall the notation $\mathcal{I}_1=\{i: p_i=1\}$ and $\mathcal{I}_2=\{i: 1<p_i\leq \infty\}$.

\begin{lema}\label{lema: equivalencia con local y global}
Let $0<\beta<mn$, $\tilde\delta\in \mathbb{R}$, $\vec{p}$ a vector of exponents and $(w,\vec{v})$ a pair of weights such that $v_i^{-1}\in\mathrm{RH}_\infty$ for $i\in\mathcal{I}_1$ and $v_i^{-p_i'}$ is doubling for $i\in\mathcal{I}_2$. Then, condition $\mathbb{H}_m(\vec{p},\beta,\tilde\delta)$ is equivalent to \eqref{eq: condicion global}.
\end{lema}

As an immediate consequence of this lemma we have the following.

\begin{coro}
	Under the hypotheses of Lemma~\ref{lema: equivalencia con local y global} we have that conditions \eqref{eq: condicion global}  implies \eqref{eq: condicion local}. 
\end{coro}

The next lemma establishes a useful property in order to give examples of weights in the considered class. We shall assume that $\beta_i=\beta/m$ for every $i$. 

\begin{lema}\label{lema: local implica global}
Let $0<\beta<mn$, $\tilde \delta<\tau=(\beta-mn)(1-1/m)+\delta/m$, $\vec{p}$ a vector of exponents and $(w,\vec{v})$ a pair of weights satisfying condition \eqref{eq: condicion local}. Then $(w,\vec{v})$ satisfies \eqref{eq: condicion global}.
\end{lema}

\begin{proof}
    The proof follows similar lines as in Lemma 2.3 in \cite{BPR22}. We shall give a scheme 
	for the sake of completeness. Let $\theta=n-\beta/m+\delta/m$. Let $B$ be a ball and $B_k=2^k B$, for $k\in \mathbb{N}$. If $i\in\mathcal{I}_1$ we have that
	\begin{align*}
	\left\|\frac{v_i^{-1}\mathcal{X}_{\mathbb{R}^n\backslash B}}{|x_B-\cdot|^{\theta}}\right\|_\infty &\leq \sum_{k=1}^\infty \left\|\frac{v_i^{-1}\mathcal{X}_{B_{k+1}\backslash B_k}}{|x_B-\cdot|^{\theta}}\right\|_\infty\\
	&\lesssim \sum_{k=1}^\infty |B_k|^{-\theta/n}\left\|v_i^{-1}\mathcal{X}_{B_{k+1}}\right\|_\infty.
	\end{align*}
	On the other hand, for $i\in\mathcal{I}_2$
\begin{align*}
\left(\int_{\mathbb{R}^n\backslash B} \frac{v_i^{-p_i'}(y)}{|x_B-y|^{\theta p_i'}}\,dy\right)^{1/p_i'}&\leq \left(\sum_{k=1}^\infty\int_{B_{k+1}\backslash B_k} \frac{v_i^{-p_i'}(y)}{|x_B-y|^{\theta p_i'}}\,dy\right)^{1/p_i'}\\
&\lesssim \sum_{k=1}^\infty |B_k|^{-\theta/n}\left(\int_{B_{k+1}}v_i^{-p_i'}\right)^{1/p_i'}.
\end{align*}
	
	By taking $\vec{k}=(k_1,k_2.\dots,k_m)$, the left-hand side of \eqref{eq: condicion global} can be bounded by a multiple constant of
	\[\sum_{\vec{k}\in\mathbb{N}^m} \prod_{i\in\mathcal{I}_1} |B_{k_i}|^{-\theta/n}\left\|v_i^{-1}\mathcal{X}_{B_{k_i+1}}\right\|_\infty\, \prod_{i\in\mathcal{I}_2} |B_{k_i}|^{-\theta/n}\left(\int_{B_{k_i+1}}v_i^{-p_i'}\right)^{1/p_i'}=\sum_{\vec{k}\in\mathbb{N}^m}I\left(B,\vec{k}\right).\]
	Observe that $\mathbb{N}^m\subset \bigcup_{i=1}^m K_i,$
	where $K_i=\{\vec{k}=(k_1,k_2,\dots,k_m): k_i\geq k_j \textrm{ for every }j\}$. Let us estimate the sum over $K_1$, being similar for the other sets. Therefore 
		\begin{align*}
	\sum_{\vec{k}\in K_1} I\left(B,\vec{k}\right)&\leq  \sum_{k_1=1}^\infty|B_{k_1}|^{-\tfrac{\theta}{n}}\prod_{i\in\mathcal{I}_1}\left\|v_i^{-1}\mathcal{X}_{B_{k_1+1}}\right\|_\infty\,\prod_{i\in\mathcal{I}_2}\left(\int_{B_{k_1+1}}v_i^{-p_i'}\right)^{1/p_i'} \prod_{i\neq 1}\sum_{k_i=1}^{k_1}|B_{k_i}|^{-\tfrac{\theta}{n}}.
	\end{align*}

	Notice that
	\[\sum_{k_i=1}^{k_1}|B_{k_i}|^{-\theta/n}=|B|^{-\theta/n}\sum_{k_i=1}^{k_1} 2^{-k_i\theta}\lesssim |B_{k_1}|^{-\theta/n}\sum_{k_i=1}^{k_1}2^{(k_1-k_i)\theta}\lesssim |B_{k_1}|^{-\theta/n}2^{k_1\theta}.\]
	
	Thus, from the estimation above  and \eqref{eq: condicion local} we obtain that
	\begin{align*}
	\frac{\|w\mathcal{X}_B\|_\infty}{|B|^{(\tilde\delta-\delta)/n}}\sum_{\vec{k}\in K_1} I\left(B,\vec{k}\right)&\lesssim |B|^{(\delta-\tilde\delta)/n}\sum_{k_1=1}^\infty 2^{(m-1)k_1\theta}\,|B_{k_1}|^{-m\theta/n}\,\|w\mathcal{X}_{B_{k_1+1}}\|_\infty\\
	&\qquad\times \prod_{i\in\mathcal{I}_1}\left\|v_i^{-1}\mathcal{X}_{B_{k_1+1}}\right\|_\infty\, \prod_{i\in\mathcal{I}_2}\left(\int_{B_{k_1+1}}v_i^{-p_i'}\right)^{1/p_i'}\\
	&\lesssim \sum_{k_1=1}^\infty 2^{-k_1(\theta-mn-\tilde\delta+\beta)},
	\end{align*}
being the last sum finite since $\tilde\delta<\tau$.
\end{proof}

The following result establishes that when we restrict the pair of weights $(w,v)$ to satisfy the relation $w=\prod_{i=1}^m v_i$, condition $\mathbb{H}_m(\vec{p},\beta,\tilde\delta)$ is equivalent, for a suitable range of the parameter $\tilde\delta$, to $A_{\vec{p},\infty}.$ 

\begin{coro}\label{coro: suficiencia clase Ap,infinito}
Let $\delta\in\mathbb{R}$, $0<\beta<mn$, $\tilde \delta<\tau=(\beta-mn)(1-1/m)+\delta/m$. Then  $\vec{v}\in \mathbb{H}_m(\vec{p},\beta,\tilde\delta)$ if and only if $\vec{v}\in A_{\vec{p},\infty}$.
\end{coro}

\begin{proof}
Let $\vec{v}\in \mathbb{H}_m(\vec{p},\beta,\tilde\delta)$. Then condition \eqref{eq: condicion local} holds. On the other hand, by Lemma~\ref{lema: consecuencia de pesos iguales}, we have that $\tilde \delta=\beta-n/p$. This implies that $\vec{v}\in A_{\vec{p},\infty}$. 

Conversely, let $\vec{v}\in  A_{\vec{p},\infty}$. Since $\tilde\delta <\tau$, by Lemma~\ref{lema: local implica global} we have that $\vec{v}$ satisfies \eqref{eq: condicion global}. By Lemma~\ref{lema: equivalencia con local y global} it will be enough to check that $v_i^{-1}\in \mathrm{RH}_\infty$ for $i\in \mathcal{I}_1$ and $v_i^{-p_i'}$ is doubling for $i\in\mathcal{I}_2$. Let us first check that $v_i^{-1}\in \mathrm{RH}_\infty$ for $i\in \mathcal{I}_1$. Let $r=p/(mp-1)$,  fix $i_0\in \mathcal{I}_1$ and observe that
	\begin{align*}
	\frac{1}{|B|}\int_B v_{i_0}^r&=\frac{1}{|B|}\int_B\prod_{i=1}^m v_i^{r}\prod_{i\neq i_0} v_i^{-r}\\
	&\leq \left\|\mathcal{X}_B\prod_{i=1}^m v_i\right\|_\infty^r\prod_{i\in\mathcal{I}_1, i\neq i_0} \|v_i^{-1}\mathcal{X}_B\|_\infty^r\prod_{i\in\mathcal{I}_2}\left(\frac{1}{|B|}\int_B v_i^{-p_i'}\right)^{r/p_i'}\\
	&\leq \frac{[\vec{v}]_{A_{\vec{p},\infty}}^r}{\|v_{i_0}^{-1}\mathcal{X}_B\|_\infty^r} \\
	&=[\vec{v}]_{A_{\vec{p},\infty}}^r\inf_B v_{i_0}^{r}.
	\end{align*}
	Therefore, $v_{i_0}^r$ is an $A_1$ weight. Then we can conclude $v_{i_0}^{-1}$ is an $\mathrm{RH}_\infty$ weight. 
	
	On the other hand, observe that $A_{\vec{p},\infty}\subseteq A_{\vec{p},q}$ for every $q>0$. If we pick $q=p$ we can apply Lemma~\ref{lema: relacion entre Ap y Ap,q} to conclude that $\vec{z}=(z_1,\dots,z_m)$ belongs to $A_{\vec{\ell}}$, where $\ell=p$, $z_i=v_i^{\ell_i}$ and $\ell_i=p_i$ for every $i$ . This implies (see, for example, Theorem 3.6 in \cite{L-O-P-T-T}) that $z_i^{1-\ell_i'}\in A_{m\ell_i'}$, that is, $v_i^{-p_i'}\in A_{mp_i'}\subseteq A_\infty$, so it is a doubling weight for every $i\in \mathcal{I}_2$. This completes the proof.
\end{proof}

The following two theorems allows us to describe the region where we can find nontrivial weights in $\mathbb{H}_m(\vec{p},\beta,\tilde\delta)$ in terms of the parameters $p$, $\beta$ and $\tilde\delta$.

\begin{teo}\label{teo: no ejemplos para la clase H_m(p,beta,delta)}
Let $\delta\in\mathbb{R}$ be fixed. Let $0<\beta<mn$,  $\tilde \delta\in \mathbb{R}$ and $\vec{p}$ be a vector of exponents. The following statements hold:
		\begin{enumerate}[\rm(a)]
		\item\label{item: teo: no ejemplos para la clase H_m(p,beta,delta) - item a} If $\tilde\delta>\delta$ or $\tilde\delta>\beta-n/p$ then condition $\mathbb{H}_m(\vec{p},\beta,\tilde\delta)$ is satisfied if and only if $v_i=\infty$ a.e. for some $1\le i\le m$.
		\item\label{item: teo: no ejemplos para la clase H_m(p,beta,delta) - item b} The same conclusion holds if $\tilde\delta=\beta-n/p=\delta$.
		\item \label{item: teo: no ejemplos para la clase H_m(p,beta,delta) - item c}If $\tilde\delta<\beta-mn$, then condition $\mathbb{H}_m(\vec{p},\beta,\tilde\delta)$ is satisfied if and only if $v_i=\infty$ a.e. for some $1\le i\le m$ or $w=0$ a.e.
		\end{enumerate}	
\end{teo}

\begin{proof}
	Let $(w,\vec{v}) \in~\mathbb{H}_m(\vec{p},\beta,\tilde\delta)$. We start with the proof of item~\eqref{item: teo: no ejemplos para la clase H_m(p,beta,delta) - item a}.  We shall first assume that 
	$\tilde\delta >\delta$. By picking a ball $B=B(x_B, R)$ such that $x_B$ is a Lebesgue point of $w^{-1}$, from \eqref{eq: condicion H_m(p,alfa,delta)} we obtain 
	\begin{align*}
	\prod_{i\in\mathcal{I}_1}\left\|\frac{v_i^{-1}}{(|B|^{1/n}+|x_B-\cdot|)^{n-\beta_i+\delta/m}}\right\|_\infty\,\prod_{i\in\mathcal{I}_2}\left(\int_{\mathbb{R}^n} \frac{v_i^{-p_i'}}{(|B|^{1/n}+|x_B-\cdot|)^{(n-\beta_i+\delta/m)p_i'}}\right)^{\tfrac{1}{p_i'}}&\lesssim \frac{|B|^{(\tilde\delta-\delta)/n}}{\|w\mathcal{X}_B\|_\infty}\\
	&\lesssim \frac{w^{-1}(B)}{|B|R^{\delta-\tilde\delta}},
	\end{align*}
	for every $R>0$. By letting $R$ approach to zero we can deduce  that there exists $1\leq i \leq m$ such that $v_i = \infty$ almost everywhere.
	
 	Let us now consider the case $\tilde\delta > \beta - n/p$. Pick a ball $B=B(x_B, R)$, being $x_B$ a Lebesgue point of $w^{-1}$ and of every $v_i^{-1}$. Then condition \eqref{eq: condicion local} implies that
 	\[\prod_{i=1}^m \frac{1}{|B|}\int_B v_i^{-1}\leq \prod_{i\in\mathcal{I}_1}\left\|v_i^{-1}\mathcal{X}_B\right\|_\infty\prod_{i\in\mathcal{I}_2}\left( \frac{1}{|B|}\int_B v_i^{-p'_i} \right)^{1/p'_i }\lesssim \frac{|B|^{\frac{\tilde\delta}{n}-\frac{\beta}{n}+\frac{1}{p}  } }{||w\mathcal{X}_B||_{\infty}}
	\lesssim \frac{w^{-1}(B)}{|B|} R^{\tilde\delta - \beta + n/p}\]
	for every $R>0$. If we let again $R$ approach to zero, we obtain
	\begin{equation*}
	\prod_{i=1 }^{m} v_{i}^{-1}(x_B)=0, 
	\end{equation*}	
	and then $\prod_{i=1 }^{m} v_{i}^{-1}$ is zero a.e. This allows us to conclude that the set $\bigcap_{i=1}^m \{v_i^{-1} >0 \}$ has null measure. Since $v_i(y)>0$ for almost every $y$ and every $i$, there exists $j$ such that $v_j = \infty$ a.e.
	
	We now proceed with item~\eqref{item: teo: no ejemplos para la clase H_m(p,beta,delta) - item b}. Suppose $\tilde\delta= \beta - n/p =\delta$. We define
	\begin{equation*}
	\xi = \sum_{i=1 }^{m }\frac{1}{p'_i} = m - \frac{1}{p}.
	\end{equation*}
	By applying Hölder inequality we obtain 
	\begin{equation*}
	\left(\int_{\mathbb{R}^n } \frac{\left(\prod_{i\in\mathcal{I}_2} v_i^{-1 }\right)^{1/\xi}}{(|B|^{1/n} + |x_B -\cdot|)^{\sum_{i\in\mathcal{I}_2}(n-\beta_i +\delta/m)/\xi }} \right)^{\xi} 
	\lesssim \prod_{i\in\mathcal{I}_2} 
	\left(
	\int_{\mathbb{R}^n } \frac{ v_i^{-p'_i }}{(|B|^{1/n} + |x_B -\cdot|)^{(n-\beta_i + \delta/m)p'_i} } \right)^{\tfrac{1}{p_i'}}
	\end{equation*}
	and since $(w,\vec{v})\in \mathbb{H}_{m}(\vec{p}, \beta, \tilde\delta)$ this implies that
	\[\prod_{i\in \mathcal{I}_1}\left\|\frac{v_i^{-1}}{(|B|^{1/n}+|x_B-\cdot|)^{n-\beta_i+\delta/m}}\right\|_\infty\left(\int_{\mathbb{R}^n } \frac{(\prod_{i\in\mathcal{I}_2} v_i^{-1 })^{1/\xi}}{(|B|^{1/n} + |x_B -\cdot|)^{\sum_{i\in\mathcal{I}_2}(n-\beta_i +\delta/m)/\xi }} \right)^{\xi}\lesssim \frac{w^{-1}(B)}{|B|},\]
	and so we can deduce that for every ball $B$
	\[\left(\int_{\mathbb{R}^n } \frac{(\prod_{i=1}^m v_i^{-1 })^{1/\xi}}{(|B|^{1/n} + |x_B -y|)^{(mn-\beta +\delta)/\xi }}\,dy\right) ^{\xi}\lesssim \frac{w^{-1}(B)}{|B|}.\]

	From this inequality we can continue by adapting an argument of \cite{Pradolini01} in order to conclude that there exists $i$ such that $v_i=\infty$ a.e. 

	We conclude with the proof of item~\eqref{item: teo: no ejemplos para la clase H_m(p,beta,delta) - item c}. If we assume that $\tilde\delta<\beta-mn$,  given a ball $B=B(x_B, R)$ and $B_0\subset B$, condition \eqref{eq: condicion local} implies that
	\begin{equation*}
	\|w\mathcal{X}_{B_0}\|_\infty \prod_{i=1}^{m}\|v_i^{-1 }\mathcal{X}_{B_0}\|_{p'_i}\leq \|w\mathcal{X}_{B}\|_\infty \prod_{i=1 }^{m}\|v_i^{-1 }\mathcal{X}_{B}\|_{p'_i}\lesssim R^{\tilde\delta -\beta+mn}.
	\end{equation*}
	The right-hand side of the inequality above tends to zero when $R$ approaches to $\infty$, which  implies that either $\|w\mathcal{X}_{B_0}\|_\infty =0$ or $\|v_i^{-1 }\mathcal{X}_{B_0}\|_{p'_i}=0$, for some  $i$. By the arbitrariness of $B_0$ we obtain either $w=0$ or $v_i =\infty$ for some $i$, respectively.
\end{proof}

\begin{teo}\label{teo: ejemplos para la clase H_m(p,beta,delta)}
	Let $\delta\in\mathbb{R}$ be fixed. Given $0<\beta<mn$, there exist pairs of weights $(w,\vec{v})$ satisfying \eqref{eq: condicion H_m(p,alfa,delta)} for every $\vec{p}$ and $\tilde\delta$ such that $\beta-mn\leq \tilde\delta\leq \min\{\delta,\beta-n/p\}$, excluding the case $\tilde\delta=\delta$ when $\beta-n/p=\delta$.
\end{teo}

The next figure depicts the area in which we can find nontrivial pair of weights belonging to $\mathbb{H}_m(\vec{p},\beta,\tilde\delta)$, for a fixed value $\delta$ and depending on $\beta$.

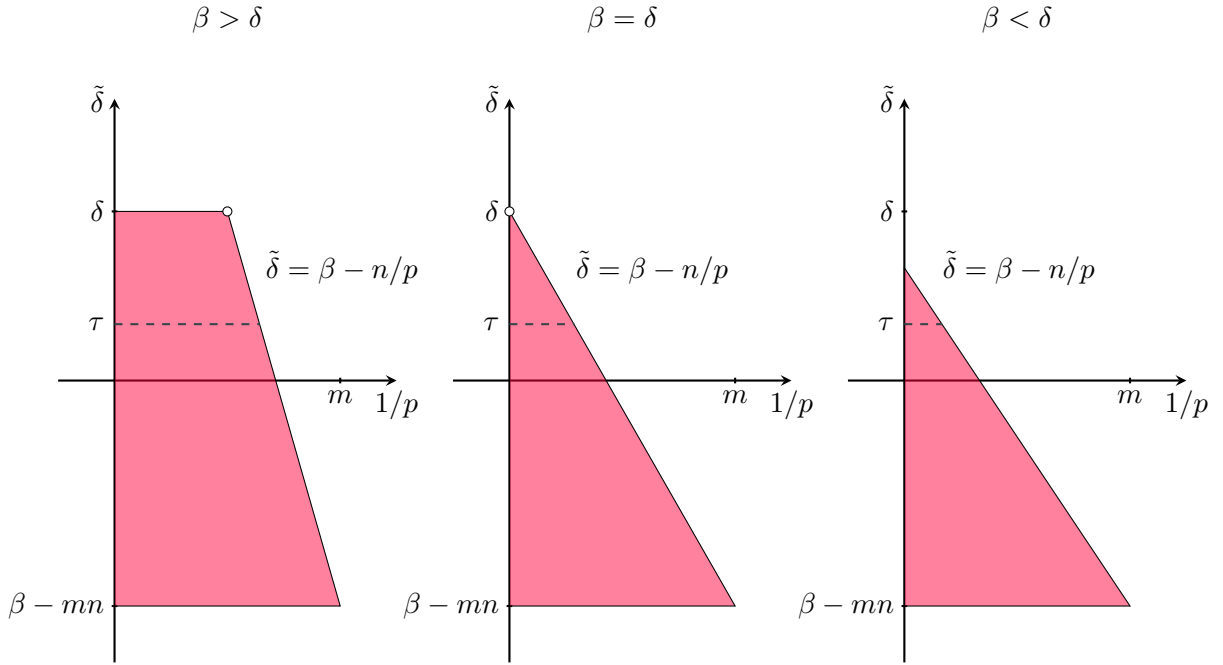
\begin{center}
	\begin{tikzpicture}[scale=0.75]
	\node[above] at (-4,6) {$\beta>\delta$};
	\draw [-stealth, thick] (-6,-5)--(-6,5);
	\draw [-stealth, thick] (-7,0)--(-1,0);
	\draw [thick] (-6.05,3)--(-5.95, 3);
	\node [left] at (-6,5) {$\tilde\delta$};
	\node [left] at (-6,3) {$\delta$};
	\node [below] at (-1,0) {$1/p$};
	\node [below] at (-2,0) {$m$};
	\draw [thick] (-2,0.05)--(-2,-0.05);
	\draw [thick] (-6.05,-4)--(-5.95, -4);
	\node [left] at (-6,-4) {$\beta-mn$};
	\draw [fill=americanrose, fill opacity=0.5] (-6,-4)--(-2,-4)--(-4,3)--(-6,3)--cycle;
	\draw [dashed, thick, color=arsenic] (-6,1)--(-3.4286,1);
	\node [left] at (-6,1) {$\tau$};
	\draw [fill=white] (-4,3) circle (0.08cm);
	\node [right] at (-3.5,2) {$\tilde\delta=\beta-n/p$};
	\node[above] at (3,6) {$\beta=\delta$};
	\draw [-stealth, thick] (1,-5)--(1,5);
	\draw [-stealth, thick] (0,0)--(6,0);
	\draw [thick] (0.95,3)--(1.05, 3);
	\node [left] at (1,5) {$\tilde\delta$};
	\node [left] at (1,3) {$\delta$};
	\node [below] at (6,0) {$1/p$};
	\node [below] at (5,0) {$m$};
	\draw [thick] (5,0.05)--(5,-0.05);
	\draw [thick] (0.95,-4)--(1.05, -4);
	\node [left] at (1,-4) {$\beta-mn$};
	\draw [fill=americanrose, fill opacity=0.5] (1,-4)--(5,-4)--(1,3)--cycle;
	\draw [dashed, thick, color=arsenic] (1,1)--(2.152857,1);
	\node [left] at (1,1) {$\tau$};
	\draw [fill=white] (1,3) circle (0.08cm);
	\node [right] at (2,2) {$\tilde\delta=\beta-n/p$};
	\node[above] at (10,6) {$\beta<\delta$};
	\draw [-stealth, thick] (8,-5)--(8,5);
	\draw [-stealth, thick] (7,0)--(13,0);
	\draw [thick] (7.95,3)--(8.05, 3);
	\node [left] at (8,5) {$\tilde\delta$};
	\node [left] at (8,3) {$\delta$};
	\node [below] at (13,0) {$1/p$};
	\node [below] at (12,0) {$m$};
	\draw [thick] (12,0.05)--(12,-0.05);
	\draw [thick] (7.95,-4)--(8.05, -4);
	\node [left] at (8,-4) {$\beta-mn$};
	\draw [fill=americanrose, fill opacity=0.5] (8,-4)--(12,-4)--(8,2)--cycle;
	\draw [dashed, thick, color=arsenic] (8,1)--(8.6667,1);
	\node [left] at (8,1) {$\tau$};
	\node [right] at (8.5,2) {$\tilde\delta=\beta-n/p$};
	\end{tikzpicture}
\end{center}

Since the classes $\mathbb{H}_m(\vec{p},\beta,\tilde\delta)$ have a similar structure as those defined in \cite{BPR22}, the proof of the theorem above will follow similar lines as that in Theorem 5.1 of \cite{BPR22}, with adequate changes. We include a sketch for the sake of completeness. We shall need the following auxiliary lemma. 
\begin{lema}\label{lema: estimacion de la integral de |x|^a en una bola}
	For a ball $B=B(x_B,R)$ in $\mathbb{R}^n$ and $\alpha>-n$, we have that
	\[\int_B |x|^{\alpha}\,dx\approx R^n\left(\max\{R,|x_B|\}\right)^\alpha.\]
\end{lema} 

\begin{proof}[Proof of Theorem~\ref{teo: ejemplos para la clase H_m(p,beta,delta)}]
Recall that $\tau=(\beta-mn)(1-1/m)+\delta/m$ is the number appearing in Lemma~\ref{lema: local implica global}, we shall split the proof into the following cases:
\begin{enumerate}[\rm(a)]
	\item \label{item: teo: ejemplos para la clase H_m(p,beta,delta) - item a}$\beta-mn<\tilde\delta<\tau\leq \beta-n/p$;
	\item \label{item: teo: ejemplos para la clase H_m(p,beta,delta) - item b}$\beta-mn<\tilde\delta\leq \beta-n/p<\tau$;
	\item \label{item: teo: ejemplos para la clase H_m(p,beta,delta) - item c}$\beta-mn<\tilde\delta=\tau<\delta<\beta-n/p$; 
	\item \label{item: teo: ejemplos para la clase H_m(p,beta,delta) - item d} $\beta-mn<\tilde\delta=\tau<\beta-n/p<\delta$;
	\item \label{item: teo: ejemplos para la clase H_m(p,beta,delta) - item e} $\tau<\tilde\delta \leq\min\{\delta,\beta-n/p\}$;
	\item \label{item: teo: ejemplos para la clase H_m(p,beta,delta) - item f}$\tilde\delta=\beta-mn$.
\end{enumerate}
Let us prove \eqref{item: teo: ejemplos para la clase H_m(p,beta,delta) - item a}. Recall that $\mathcal{I}_1=\{i: p_i=1\}$, $\mathcal{I}_2=\{i: p_i>1\}$ and let $m_j=\#\mathcal{I}_j$, for $j=1,2$. Since  $m_1<m$ by the restrictions on the parameters, we can take
\[0<\varepsilon<\frac{mn-\beta+\tilde\delta}{m-m_1}.\]
For $1\leq i\leq m$ we define
\[\xi_i=\left\{\begin{array}{ccl}
0 & \textrm{ if } & i\in\mathcal{I}_1,\\
\frac{n}{p_i'}-\varepsilon & \textrm{ if } & i\in\mathcal{I}_2.
\end{array}
\right.
\]
 Let $\rho=\sum_{i=1}^m \xi_i+\tilde\delta-\beta+n/p>0$. Then we take
\[w(x)=|x|^{\rho}\quad\textrm{ and }\quad v_{i}(x)=|x|^{\xi_i}.\]

By virtue of Lemma~\ref{lema: local implica global} it will be enough to show that $(w,\vec{v})$ verifies condition \eqref{eq: condicion local}. Let $B=B(x_B, R)$ and assume that $|x_B|\leq R$. If $i\in\mathcal{I}_2$, by Lemma~\ref{lema: estimacion de la integral de |x|^a en una bola} we get
\begin{equation*}
\left(\frac{1}{|B|} \int_B v_i^{-p'_i}
\right)^{1/{p_i'}}=
\left(\frac{1}{|B|} \int_B |x|^{-\xi_ip'_i}\, dx
\right)^{1/{p_i'}}\approx R^{-\xi_i},
\end{equation*}
and $\left\|v_i^{-1}\mathcal{X}_B\right\|_\infty=1$ for $i\in\mathcal{I}_1$.
On the other hand, $\|w\mathcal{X}_B\|_{\infty}\lesssim R^\rho$
since $\rho>0$. Therefore,  
\[\frac{\|w\mathcal{X}_B\|_{\infty}}{|B|^{\tilde\delta/n-\beta/n+1/p}}\prod_{i\in\mathcal{I}_1}\left\|v_i^{-1}\mathcal{X}_B\right\|_{\infty}\prod_{i\in \mathcal{I}_2}\left(\frac{1}{|B|}\int_B v_i^{-p_i'}\right)^{1/p_i'}\lesssim R^{\rho -\sum_{i=1 }^{m}\xi_i-\tilde\delta +\beta -n/p}\leq C.\]

We now consider the case $|x_B|>R$. We have that 
\begin{equation*}
\|w\mathcal{X}_B\|_{\infty} \lesssim |x_B|^{\rho},
\end{equation*}
whilst for $i\in\mathcal{I}_2$	
\begin{equation*}
\left(\frac{1}{|B|} \int_B |x|^{-\xi_i p'_i}\,dx \right)^{1/{p_i'}} \approx |x_B|^{-\xi_i}.
\end{equation*}
Consequently, since $\tilde\delta<\beta-n/p$ we get
\begin{equation*}
\frac{\|w\mathcal{X}_B\|_{\infty}}{|B|^{\tilde\delta/n-\beta/n+1/p}}\prod_{i\in\mathcal{I}_1}\left\|v_i^{-1}\mathcal{X}_B\right\|_{\infty}\prod_{i\in \mathcal{I}_2}\left(\frac{1}{|B|}\int_B v_i^{-p_i'}\right)^{1/p_i'}\lesssim
|x_B|^{\rho - \sum_{i=1}^m \xi_i -\tilde\delta +\beta -n/p} \leq C,
\end{equation*}
which completes the proof of \eqref{item: teo: ejemplos para la clase H_m(p,beta,delta) - item a}. 

We now prove \eqref{item: teo: ejemplos para la clase H_m(p,beta,delta) - item b}. In this case we take $w=1$ and $v_i=|x|^{\xi_i}$,  $\xi_i=(\beta-\tilde\delta)/m-n/p_i$ for every $1\leq i\leq m$. By Lemma~\ref{lema: local implica global} it will be enough to prove that $(w,\vec{v})$ satisfies condition \eqref{eq: condicion local}. Pick a ball $B=B(x_B,R)$ and assume that $|x_B|\le R$. If $i\in\mathcal{I}_1$ we get $\xi_i<0$, since we are assuming $\tilde\delta>\beta-mn$. In this case we get
\[\left\|v_i^{-1}\mathcal{X}_B\right\|_\infty\approx R^{-\xi_i}.\]
On the other hand, for $i\in\mathcal{I}_2$ we have $\xi_i<n/p_i'$, so Lemma~\ref{lema: estimacion de la integral de |x|^a en una bola} yields
\[\left(\frac{1}{|B|}\int_B v_i^{-p_i'}\right)^{1/p_i'}\approx R^{-\xi_i}.\]
These two estimates imply that
\[\frac{\|w\mathcal{X}_B\|_{\infty}}{|B|^{\tilde\delta/n-\beta/n+1/p}}\prod_{i\in\mathcal{I}_1}\left\|v_i^{-1}\mathcal{X}_B\right\|_{\infty}\prod_{i\in \mathcal{I}_2}\left(\frac{1}{|B|}\int_B v_i^{-p_i'}\right)^{1/p_i'}\lesssim \frac{R^{-\sum_{i=1}^m\xi_i}}{R^{\tilde\delta-\beta+n/p}}=1.\]
If $|x_B|>R$, we have that $\left\|v_i^{-1}\mathcal{X}_B\right\|_{\infty}\lesssim |x_B|^{-\xi_i}$ and also 
\[\left(\frac{1}{|B|}\int_B v_i^{-p_i'}\right)^{1/p_i'}\approx |x_B|^{-\xi_i}\]
by Lemma~\ref{lema: estimacion de la integral de |x|^a en una bola}. Thus
\[\frac{\|w\mathcal{X}_B\|_{\infty}}{|B|^{\tilde\delta/n-\beta/n+1/p}}\prod_{i\in\mathcal{I}_1}\left\|v_i^{-1}\mathcal{X}_B\right\|_{\infty}\prod_{i\in \mathcal{I}_2}\left(\frac{1}{|B|}\int_B v_i^{-p_i'}\right)^{1/p_i'}\lesssim \frac{|x_B|^{-\sum_{i=1}^m\xi_i}}{R^{\tilde\delta-\beta+n/p}}\leq 1,\]
since $\tilde\delta\leq\beta-n/p$. This concludes the proof of item~\eqref{item: teo: ejemplos para la clase H_m(p,beta,delta) - item b}. 

In order to prove \eqref{item: teo: ejemplos para la clase H_m(p,beta,delta) - item c} we pick $(\beta-\tau)/m-n/p_i<\xi_i<n/p_i'$ for $i\in\mathcal{I}_2$ and $\xi_i=0$ for $i\in\mathcal{I}_1$.  We also take $\rho=\sum_{i=1}^m\xi_i+\tau-\beta+n/p>0$ and define $w(x)=|x|^\rho$ and $v_i(x)=|x|^{\xi_i}$, for $1\leq i\leq m$. We first notice that
\[\prod_{i\in\mathcal{I}_1}\left\|\frac{v_i^{-1}\mathcal{X}_{\mathbb{R}^n\backslash B}}{|x_B-\cdot|^{n-\beta/m+\delta/m}}\right\|_\infty\leq R^{-\sum_{i\in\mathcal{I}_1}(n-\beta/m+\delta/m)}.\]

By virtue of Lemma~\ref{lema: equivalencia con local y global} we have to prove that condition \eqref{eq: condicion global} holds. By using the estimate above, it will be enough to show that

\begin{equation}\label{eq: teo: ejemplos para Hbb - item c - eq1}
R^{\delta-\tilde\delta-\sum_{i\in\mathcal{I}_1}(n-\beta/m+\delta/m)}\|w\mathcal{X}_B\|_\infty\prod_{i\in\mathcal{I}_2}\left(\int_{\mathbb{R}^n\backslash B} \frac{v_i^{-p_i'}(y)}{|x_B-y|^{(n-\beta/m+\delta/m)p_i'}}\,dy\right)^{1/p_i'}\leq C
\end{equation}
for every ball $B$. We shall first assume that $|x_B|\leq R$. Let $B_k=B\left(x_B, 2^kR\right)$ for $k\in\mathbb{N}_0$ and $i\in\mathcal{I}_2$. By Lemma~\ref{lema: estimacion de la integral de |x|^a en una bola}, we get \refstepcounter{BPR}\label{pag: estimacion para i fuera de I_1, |x_B|<=R}
\begin{align*}
\left(\int_{\mathbb{R}^n\backslash B} \frac{v_i^{-p_i'}(y)}{|x_B-y|^{(n-\beta/m+\delta/m)p_i'}}\,dy\right)^{1/p_i'}&\lesssim \sum_{k=0}^\infty (2^kR)^{-n+\beta/m-\delta/m}\left(\int_{B_{k+1}\backslash B_k} |y|^{-\xi_ip_i'}\,dy\right)^{1/p_i'}\\
&\lesssim \sum_{k=0}^\infty (2^kR)^{-n+\beta/m-\delta/m-\xi_i+n/p_i'}\\
&\lesssim R^{-n/p_i+\beta/m-\delta/m-\xi_i},
\end{align*}
since $-n/p_i+\beta/m-\delta/m-\xi_i<0$ by the election of $\xi_i$. Since $\tilde\delta =\tau$, the left-hand side of \eqref{eq: teo: ejemplos para Hbb - item c - eq1} is bounded by a multiple constant of
\[R^{\delta-\tilde\delta-\sum_{i\in\mathcal{I}_1}(n-\beta/m+\delta/m)+\rho-\sum_{i\in\mathcal{I}_2}(n/p_i-\beta/m+\delta/m+\xi_i)}= R^{-\tilde\delta-n/p+\beta+\rho-\sum_{i=1}^m\xi_i}= 1.\]

We now assume $|x_B|>R$. There exists a number $N$ such that $2^NR<|x_B|\leq 2^{N+1}R$. If $i\in\mathcal{I}_2$ we have that \refstepcounter{BPR}\label{pag: estimacion del producto para i fuera de I_1, |x_B|>R}
\begin{align*}
\left(\int_{\mathbb{R}^n\backslash B} \frac{v_i^{-p_i'}(y)}{|x_B-y|^{(n-\beta/m+\delta/m)p_i'}}\,dy\right)^{1/p_i'}&\lesssim \sum_{k=0}^\infty(2^{k}R)^{-n+\beta/m-\delta/m}\left(\int_{B_k} |y|^{-\xi_ip_i'}\,dy\right)^{1/p_i'}\\
&=\sum_{k=0}^N+\sum_{k=N+1}^\infty=S_1^i+S_2^i.
\end{align*}
Let $\theta_i=n/p_i+(\delta-\beta)/m$, for $1\leq i\leq m$.  We shall first prove that if $\theta_i<0$, then
\begin{equation}\label{eq: teo: ejemplos para Hbb - item c - eq2}
S_j^i\lesssim|x_B|^{-\xi_i-\theta_i},
\end{equation} 
for $j=1,2$. Indeed, by Lemma~\ref{lema: estimacion de la integral de |x|^a en una bola} we obtain
\begin{align*}
S_1^i&\lesssim \sum_{k=0}^N(2^{k}R)^{-n+\beta/m-\delta/m+n/p_i'}|x_B|^{-\xi_i}\\
&\lesssim |x_B|^{-\xi_i}R^{-\theta_i}\sum_{k=0}^N 2^{-k\theta_i}\\
&\lesssim |x_B|^{-\xi_i}(2^NR)^{-\theta_i}\\
&\lesssim |x_B|^{-\xi_i-\theta_i}.
\end{align*}
 For $S_2^i$ we apply again Lemma~\ref{lema: estimacion de la integral de |x|^a en una bola} in order to get
\begin{align*}
S_2^i&\lesssim \sum_{k=N+1}^\infty\left(2^{k}R\right)^{-n+\beta/m-\delta/m+n/p_i'-\xi_i}\lesssim \sum_{k=N+1}^\infty \left(2^{k}R\right)^{-\xi_i-\theta_i}\\
&= \left(2^{N+1}R\right)^{-\xi_i-\theta_i}\sum_{k=0}^\infty 2^{-k(\xi_i+\theta_i)}\lesssim |x_B|^{-\xi_i-\theta_i},
\end{align*}
since $\theta_i+\xi_i=n/p_i+(\delta-\beta)/m+\xi_i>0$.

We now assume that $\theta_i=0$. By proceeding similarly as in the previous case, we have
\[S_1^i\lesssim |x_B|^{-\xi_i}N\lesssim |x_B|^{-\xi_i}\log_2\left(\frac{|x_B|}{R}\right),\]
and
\[S_2^i\lesssim |x_B|^{-\xi_i}\]
since $\xi_i>0$ when $\theta_i=0$. Consequently,
\begin{equation}\label{eq: teo: ejemplos para Hbb - item c - eq3}
S_1^i+S_2^i\lesssim |x_B|^{-\xi_i}\left(1+\log_2\left(\frac{|x_B|}{R}\right)\right)\lesssim |x_B|^{-\xi_i}\log_2\left(\frac{|x_B|}{R}\right).
\end{equation}

We finally consider the case $\theta_i>0$. For $S_2^i$ we can proceed exactly as in the case $\theta_i<0$ and get the same bound. On the other hand, for $S_1^i$ we have that \refstepcounter{BPR}\label{pag: estimacion de S_1^i y S_2^i,  theta_i>0}
\[S_1^i\lesssim \sum_{k=0}^N(2^{k}R)^{-n+\beta/m-\delta/m+n/p_i'}|x_B|^{-\xi_i}
\lesssim |x_B|^{-\xi_i}R^{-\theta_i}\sum_{k=0}^N 2^{-k\theta_i}
\lesssim |x_B|^{-\xi_i-\theta_i}2^{N\theta_i}.\]
Therefore, if $i\in\mathcal{I}_2$ and $\theta_i>0$ we get
\begin{equation}\label{eq: teo: ejemplos para Hbb - item c - eq4}
S_1^i+S_2^i\lesssim |x_B|^{-\xi_i-\theta_i}\left(1+2^{N\theta_i}\right)\lesssim 2^{N\theta_i}|x_B|^{-\xi_i-\theta_i}. 
\end{equation}
By combining \eqref{eq: teo: ejemplos para Hbb - item c - eq2}, \eqref{eq: teo: ejemplos para Hbb - item c - eq3} and \eqref{eq: teo: ejemplos para Hbb - item c - eq4} we obtain
\begin{align*}
\prod_{i\in\mathcal{I}_2}\left(\int_{\mathbb{R}^n\backslash B} \frac{v_i^{-p_i'}(y)}{|x_B-y|^{(n-\beta/m+\delta/m)p_i'}}\,dy\right)^{1/p_i'}&\lesssim \prod_{i\in\mathcal{I}_2, \theta_i<0} |x_B|^{-\xi_i-\theta_i} \prod_{i\in\mathcal{I}_2, \theta_i=0} |x_B|^{-\xi_i}\log_2\left(\frac{|x_B|}{R}\right) \\
&\quad\times\prod_{i\in\mathcal{I}_2, \theta_i>0} |x_B|^{-\xi_i-\theta_i}\,2^{N\theta_i} \\
&\lesssim |x_B|^{-\sum_{i\in\mathcal{I}_2}(\xi_i+\theta_i)}\,2^{N\sum_{i\in\mathcal{I}_2,\theta_i> 0}\theta_i}\\
&\quad \times\left(\log_2\left(\frac{|x_B|}{R}\right)\right)^{\#\{i\in\mathcal{I}_2, \theta_i=0\}},
\end{align*}
so the left-hand side of \eqref{eq: teo: ejemplos para Hbb - item c - eq1} can be bounded by a multiple constant of
\[R^{\delta-\tilde\delta-(n-\beta/m+\delta/m)m_1}|x_B|^{\rho-\sum_{i\in\mathcal{I}_2}(\xi_i+\theta_i)}\,2^{N\sum_{i\in\mathcal{I}_2,\theta_i> 0}\theta_i}\left(\log_2\left(\frac{|x_B|}{R}\right)\right)^{\#\{i\in\mathcal{I}_2, \theta_i=0\}}\]
which is equivalent to
\begin{equation}\label{eq: teo: ejemplos para Hbb - item c - eq5}
\left(\frac{|x_B|}{R}\right)^{\tau-\delta+(n-\beta/m+\delta/m)m_1+\sum_{i\in\mathcal{I}_2,\theta_i> 0}\theta_i}\left(\log_2\left(\frac{|x_B|}{R}\right)\right)^{\#\{i\in\mathcal{I}_2, \theta_i=0\}}.
\end{equation}
Since $\theta_i<n+(\delta-\beta)/m$ for $i\in\mathcal{I}_2$, there exists $\varepsilon>0$ that verifies \refstepcounter{BPR}\label{pag: estimacion del log_2(|x_B|/R)}
\[\sum_{i\in\mathcal{I}_2,\theta_i> 0}\theta_i+\varepsilon\#\{i\in\mathcal{I}_2, \theta_i=0\}\leq \left(n+\frac{\delta-\beta}{m}\right)\#\{i\in\mathcal{I}_2, \theta_i>0\}.\]
By using the fact that $\log_2 t\lesssim \varepsilon^{-1}t^{\varepsilon}$ for every $t\geq 1$, we can majorize \eqref{eq: teo: ejemplos para Hbb - item c - eq5} by a constant factor provided that
\[\tau-\delta+\left(n+\frac{\delta-\beta}{m}\right)\left(m_1+\#\{i\in\mathcal{I}_2: \theta_i>0\}\right)\leq \tau-\delta+\left(n+\frac{\delta-\beta}{m}\right)(m-1)=0.\]
Indeed, if this last inequality did not hold, then we would have that $\theta_i>0$ for every $i\in\mathcal{I}_2$. We also observe that $\theta_i>0$ for $i\in\mathcal{I}_1$. This would lead to $n/p>\beta-\delta$, a contradiction.

In order to prove \eqref{item: teo: ejemplos para la clase H_m(p,beta,delta) - item d} we only consider two cases. If there exists some $i\in\mathcal{I}_2$ such that $\theta_i\leq 0$, the proof follows exactly as in \eqref{item: teo: ejemplos para la clase H_m(p,beta,delta) - item c}. If not, that is $\theta_i>0$ for every  $i\in\mathcal{I}_2$, observe that
\begin{align*}
\tau-\delta+\left(n+\frac{\delta-\beta}{m}\right)m_1+\sum_{i\in\mathcal{I}_2,\theta_i> 0}\theta_i&=\tau-\delta+\sum_{i=1}^m \left(\frac{n}{p_i}+\frac{\delta-\beta}{m}\right)\\
&=\tau+\frac{n}{p}-\beta\\
&<0,
\end{align*}
then we can choose $\varepsilon>0$ small enough so that the resulting exponent for $|x_B|/R$ in \eqref{eq: teo: ejemplos para Hbb - item c - eq5} becomes negative. 

We now proceed with the proof of \eqref{item: teo: ejemplos para la clase H_m(p,beta,delta) - item e}. Let us first suppose that $\tilde\delta<\min\{\delta,\beta-n/p\}$. We take
$\rho=\tilde\delta-\tau>0$ and $\xi_i=(\delta-\tau)/m-\theta_i$, for every $i$. Then we define $w(x)=|x|^\rho$ and $v_i=|x|^{\xi_i}$, $1\leq i\leq m$. These functions are locally integrable since $\rho>0$ and $\xi_i<n/p_i'$. Furthermore, $\xi_i<0$ for $i\in\mathcal{I}_1$, so $v_i^{-1}\in \mathrm{RH}_\infty$ for these index. Then, by Lemma~\ref{lema: equivalencia con local y global}, it will be enough to show that condition \eqref{eq: condicion global} holds. Fix a ball $B=B(x_B, R)$ and assume that $|x_B|<R$. Then we get
\begin{equation}\label{eq: teo: ejemplos para Hbb - item e - eq1}
\frac{\|w\mathcal{X}_B\|_\infty}{|B|^{(\tilde\delta-\delta)/n}}\lesssim R^{\delta-\tilde\delta+\rho}=R^{\delta-\tau}.
\end{equation}

On the other hand, if $i\in\mathcal{I}_1$ we have
\begin{align*}
\left\|\frac{v_i^{-1}\mathcal{X}_{\mathbb{R}^n\backslash B}}{|x_B-\cdot|^{n-\beta/m+\delta/m}}\right\|_\infty&\lesssim \sum_{k=0}^\infty \left\|\frac{v_i^{-1}\mathcal{X}_{B_{k+1}\backslash B_k}}{|x_B-\cdot|^{n-\beta/m+\delta/m}}\right\|_\infty\\
&\lesssim \sum_{k=0}^\infty \left(2^kR\right)^{-\xi_i-n+\beta/m-\delta/m}\\
&\lesssim   R^{(\tau-\delta)/m},
\end{align*}
since $\tau<\tilde\delta<\delta$. This yields
\begin{equation}\label{eq: teo: ejemplos para Hbb - item e - eq2}
\prod_{i\in\mathcal{I}_1}\left\|\frac{v_i^{-1}\mathcal{X}_{\mathbb{R}^n\backslash B}}{|x_B-\cdot|^{n-\beta/m+\delta/m}}\right\|_\infty\lesssim R^{m_1(\tau-\delta)/m}.
\end{equation}
Finally, since $\xi_i+\theta_i=(\delta-\tau)/m>0$ for $i\in\mathcal{I}_2$, we can proceed as in page~\pageref{pag: estimacion para i fuera de I_1, |x_B|<=R} to obtain
\begin{equation}\label{eq: teo: ejemplos para Hbb - item e - eq3}
\prod_{i\in\mathcal{I}_2}\left(\int_{\mathbb{R}^n\backslash B}\frac{v_i^{-p_i'}(y)}{|x_B-y|^{(n-\beta/m+\delta/m)p_i'}}\,dy\right)^{1/p_i'}\lesssim R^{m_2(\tau-\delta)/m}.
\end{equation}
By combining \eqref{eq: teo: ejemplos para Hbb - item e - eq1}, \eqref{eq: teo: ejemplos para Hbb - item e - eq2} and \eqref{eq: teo: ejemplos para Hbb - item e - eq3}, the left-hand side of \eqref{eq: condicion global} is bounded by a constant $C$.

We now consider the case $|x_B|>R$. We have that
\begin{equation}\label{eq: teo: ejemplos para Hbb - item e - eq4}
\frac{\|w\mathcal{X}_B\|_\infty}{|B|^{(\tilde\delta-\delta)/n}}\lesssim R^{\delta-\tilde\delta}|x_B|^\rho.
\end{equation}
Since $|x_B|>R$, there exists a number $N\in\mathbb{N}$ such that $2^NR<|x_B|\leq 2^{N+1}R$. For $i\in\mathcal{I}_1$ we write
\begin{align*}
\left\|\frac{v_i^{-1}\mathcal{X}_{\mathbb{R}^n\backslash B}}{|x_B-\cdot|^{n-\beta/m+\delta/m}}\right\|_\infty&\lesssim \sum_{k=0}^N \left\|\frac{v_i^{-1}\mathcal{X}_{B_{k+1}\backslash B_k}}{|x_B-\cdot|^{n-\beta/m+\delta/m}}\right\|_\infty+\sum_{k=N+1}^\infty \left\|\frac{v_i^{-1}\mathcal{X}_{B_{k+1}\backslash B_k}}{|x_B-\cdot|^{n-\beta/m+\delta/m}}\right\|_\infty\\
&=S_1^i+S_2^i.
\end{align*}
By proceeding as in page~\pageref{pag: estimacion de S_1^i y S_2^i,  theta_i>0} with $p_i=1$ we have that
\begin{equation}\label{eq: teo: ejemplos para Hbb - item e - eq5}
\prod_{i\in\mathcal{I}_1}\left\|\frac{v_i^{-1}\mathcal{X}_{\mathbb{R}^n\backslash B}}{|x_B-\cdot|^{n-\beta/m+\delta/m}}\right\|_\infty\lesssim |x_B|^{-\sum_{i\in\mathcal{I}_1}(\xi_i+\theta_i)}2^{N\sum_{i\in\mathcal{I}_1}\theta_i}.
\end{equation}
Finally, if $i\in\mathcal{I}_2$ our choice of $\xi_i$ allows us to follow the argument given in page~\pageref{pag: estimacion del producto para i fuera de I_1, |x_B|>R} to conclude that
\begin{align*}
\prod_{i\in\mathcal{I}_2}\left(\int_{\mathbb{R}^n\backslash B} \frac{v_i^{-p_i'}(y)}{|x_B-y|^{(n-\beta/m+\delta/m)p_i'}}\,dy\right)^{1/p_i'}
&\lesssim |x_B|^{-\sum_{i\in\mathcal{I}_2}(\xi_i+\theta_i)}2^{N\sum_{i\in\mathcal{I}_2,\theta_i> 0}\theta_i}\\
&\quad \times\left(\log_2\left(\frac{|x_B|}{R}\right)\right)^{\#\{i\in\mathcal{I}_2, \theta_i=0\}}.
\end{align*}
By combining the inequality above with \eqref{eq: teo: ejemplos para Hbb - item e - eq4} and \eqref{eq: teo: ejemplos para Hbb - item e - eq5}, the left-hand side of \eqref{eq: condicion global} can be bounded by a multiple constant of
\[R^{\delta-\tilde\delta} |x_B|^{\rho-\sum_{i=1}^m(\theta_i+\xi_i)}2^{N\sum_{i:\theta_i> 0}\theta_i}
\left(\log_2\left(\frac{|x_B|}{R}\right)\right)^{\#\{i\in\mathcal{I}_2, \theta_i=0\}}\]
which is equal to
\[\left(\frac{R}{|x_B|}\right)^{\delta-\tilde\delta-\sum_{i: \theta_i>0}\theta_i}\left(\log_2\left(\frac{|x_B|}{R}\right)\right)^{\#\{i\in\mathcal{I}_2, \theta_i=0\}}.\]
If $\theta_i<0$ for every $i$ then the exponent of $R/|x_B|$ is positive. On the other hand, if $\theta_i\geq 0$ for every $i$ then
\[\delta-\tilde\delta-\sum_{i: \theta_i>0}\theta_i=\delta-\tilde\delta-\sum_{i=1}^m \theta_i=\delta-\tilde\delta-\frac{n}{p}+\beta-\delta>0,\]
since $\tilde\delta<\beta-n/p$. In both cases we can repeat the argument given in page~\pageref{pag: estimacion del log_2(|x_B|/R)} to conclude that $(w,\vec{v})$ belongs to $\mathbb{H}_m(\vec{p},\beta,\tilde\delta)$. Let us observe that, for example, if $\beta\leq\delta$ then every $\theta_i$ is nonnegative.

If $\tilde\delta=\delta<\beta-n/p$ or $\tilde\delta=\beta-n/p<\delta$ the same estimation as above works when we take $\theta_i<0$ for every $i$. The second case also works when $\theta_i>0$ for every $i$.

We finish with the proof of item~\eqref{item: teo: ejemplos para la clase H_m(p,beta,delta) - item f}. In this case we fix $\rho>0$ and take $w(x)=\left(1+|x|^\rho\right)^{-m_1}$. If $g_i$ are nonnegative fixed functions in $L^{p_i'}(\mathbb{R}^n)$ for $i\in\mathcal{I}_2$, we define
\[v_i(x)=\left\{
\begin{array}{ccl}
e^{|x|}&\textrm{ if }& i\in\mathcal{I}_1,\\
g_i^{-1}&\textrm{ if }& i\in\mathcal{I}_2.
\end{array}
\right.\] 
Fix a ball $B=B(x_B, R)$. It is enough to check condition \eqref{eq: condicion local}, since $\tilde\delta=\beta-mn<\tau$. Notice that 
\[\left\|w\mathcal{X}_B\right\|_\infty\prod_{i\in\mathcal{I}_1}\left\|v_i^{-1}\mathcal{X}_B\right\|_\infty\leq \prod_{i\in\mathcal{I}_1}\left\|(1+|\cdot|^{\rho})^{-1}\mathcal{X}_B\right\|_\infty\left\|e^{-|\cdot|}\mathcal{X}_B\right\|_\infty\leq 1.\]
Therefore,
\begin{align*}
\frac{\|w\mathcal{X}_B\|_{\infty}}{|B|^{\tilde\delta/n-\beta/n+1/p}}\prod_{i\in\mathcal{I}_1}\left\|v_i^{-1}\mathcal{X}_B\right\|_{\infty}\prod_{i\in \mathcal{I}_2}\left(\frac{1}{|B|}\int_B v_i^{-p_i'}\right)^{1/p_i'}
&\lesssim \prod_{i\in\mathcal{I}_2}\|g_i\|_{p_i'},
\end{align*}
for every ball $B$. This concludes the proof of \eqref{item: teo: ejemplos para la clase H_m(p,beta,delta) - item f}.
\end{proof}


\def\cprime{$'$}
\providecommand{\bysame}{\leavevmode\hbox to3em{\hrulefill}\thinspace}
\providecommand{\MR}{\relax\ifhmode\unskip\space\fi MR }
\providecommand{\MRhref}[2]{%
  \href{http://www.ams.org/mathscinet-getitem?mr=#1}{#2}
}
\providecommand{\href}[2]{#2}

\end{document}